\documentclass[10.5pt]{article}

\usepackage{amssymb}
\usepackage{amsmath,amsfonts,mathtools}
\usepackage{amsthm}
\usepackage{latexsym}
\usepackage{mathrsfs}
\usepackage{color}
\usepackage{float}
\usepackage{enumitem}
\usepackage{algorithm,algorithmic}
\usepackage{soul}
 \allowdisplaybreaks[3]

\usepackage[a4paper]{geometry}
\newcommand{\beq}{\begin{equation}}
\newcommand{\eeq}{\end{equation}}
\newcommand{\beqa}{\begin{eqnarray}}
\newcommand{\eeqa}{\end{eqnarray}}
\newcommand{\beqas}{\begin{eqnarray*}}
\newcommand{\eeqas}{\end{eqnarray*}}
\newcommand{\ba}{\begin{array}}
\newcommand{\ea}{\end{array}}
\newcommand{\bi}{\begin{itemize}}
\newcommand{\ei}{\end{itemize}}
\newcommand{\nn}{\nonumber}

\DeclareMathOperator*{\argmin}{argmin}  

\newcommand{\mcX}{{\mathcal X}}
\newcommand{\mcY}{{\mathcal Y}}

\newcommand{\prox}{\mathrm{prox}}
\newcommand{\dom}{\mathrm{dom}}
\newcommand{\dist}{\mathrm{dist}}

\newtheorem{lemma}{Lemma}
\newtheorem{thm}{Theorem}

\newtheorem{defi}{Definition}

\newtheorem{assumption}{Assumption}

\newtheorem{rem}{Remark}

\newcounter{spb}
\def\bfx{{\rm \bf x}}
\def\bfy{{\rm \bf y}}
\def\cA{{\cal A}}
\def\cBr{{\mathbb B^+_r}}

\def\cL{{\cal L}}
\def\calN{{\cal N}}

\def\cO{{\cal O}}
\def\cS{{\cal S}}

\def\tlambda{{\tilde \lambda}}
\def\cI{{\mathscr I}}

\def\tg{{\tilde g}}

\def\tL{{\widetilde L}}

\def\tP{{\widetilde P}}

\def\tx{{\hat x}}
\def\ty{{\hat y}}

\def\bR{{\mathbb{R}}}
\def\hz{{\hat z}}
\def\tf{{\tilde f}}

\def\xe{{x_\varepsilon}}
\def\ye{{y_\varepsilon}}
\def\ze{{z_\varepsilon}}

\def\bh{{h}}
\def\bH{{H}}

\def\bk{{\bar k}}
\def\h{\bar h}

\def\hL{{\widehat L}}

\def\Cr{{\delta}}

\definecolor{ngreen}{RGB}{38,217,169}

\title{First-order penalty methods for bilevel optimization
}
\author{
Zhaosong Lu
\thanks{
Department of Industrial and Systems Engineering, University of Minnesota, USA (email: {\tt zhaosong@umn.edu}, {\tt mei00035@umn.edu}). This work was partially supported by NSF Award IIS-2211491.}
\and
Sanyou Mei
\footnotemark[1]
}

\date{January 4, 2023 (Revised: April 18, 2023; September 27, 2023; December 6, 2023)}

\begin{document}
\maketitle

\begin{abstract}
In this paper we study a class of unconstrained and constrained bilevel optimization problems in which the lower level is a possibly nonsmooth convex optimization problem, while the upper level is a possibly nonconvex optimization problem. We introduce a notion of $\varepsilon$-KKT solution for them and show that  an $\varepsilon$-KKT solution leads to an $\cO(\sqrt{\varepsilon})$- or $\cO(\varepsilon)$-hypergradient based stionary point under suitable assumptions. We also propose first-order penalty methods for finding an $\varepsilon$-KKT solution of them, whose subproblems turn out to be a structured minimax problem and can be suitably solved by a first-order method recently developed by the authors. Under suitable assumptions, an \emph{operation complexity} of $\cO(\varepsilon^{-4}\log\varepsilon^{-1})$ and $\cO(\varepsilon^{-7}\log\varepsilon^{-1})$,  measured by their fundamental operations, is established for the proposed penalty methods for finding an $\varepsilon$-KKT solution of the unconstrained and constrained bilevel optimization problems, respectively. Preliminary numerical results are presented to illustrate the performance of our proposed methods. To the best of our knowledge, this paper is the first work to demonstrate that bilevel optimization can be approximately solved as minimax optimization, and moreover, it provides the first implementable method with complexity guarantees for such sophisticated bilevel optimization.
\end{abstract}


\medskip


\section{Introduction}

Bilevel optimization is a two-level hierarchical optimization in which the decision variables in the upper level are also involved in the lower level. Generically, it can be written in the following form:
 \begin{equation}\label{BLO}
\begin{array}{rl}
\min\limits_{x,y} & f(x,y)\\
\mbox{s.t.}& g(x,y)\leq 0, \ \ y\in\argmin\limits_z\{\tf(x,z)|\tg(x,z)\leq0\}.\footnote{}                  
\end{array}
\end{equation}
\footnotetext{For ease of reading, throughout this paper the tilde symbol is particularly used for the functions related to the lower-level optimization problem. Besides,  ``$\argmin$'' denotes the set of optimal solutions of the associated problem.}
Bilevel optimization has found a variety of important applications, including adversarial training \cite{Madry18,mirrlees1999theory,szegedy2013intriguing},   continual learning \cite{lopez2017gradient}, hyperparameter tuning \cite{bennett2008bilevel,franceschi2018bilevel}, image reconstruction \cite{crockett2022bilevel}, meta-learning \cite{bertinetto2018meta,ji2020convergence,rajeswaran2019meta}, neural architecture search \cite{feurer2019hyperparameter,liu2018darts}, reinforcement learning \cite{hong2023two,konda1999actor}, and Stackelberg games \cite{von2010market}. More applications about it can be found in \cite{bard2013practical,colson2007overview,dempe2002foundations,dempe2015bilevel,
dempe2020bilevel,shimizu2012nondifferentiable} and the references therein. 
Theoretical properties including optimality conditions of \eqref{BLO} have been extensively
studied in the literature (e.g., see \cite{dempe2020bilevel,dempe2013bilevel,ma2021combined,vicente1994bilevel,ye2020constraint}).  

Numerous methods have been developed for solving some special cases of \eqref{BLO}. For example,  constraint-based methods \cite{hansen1992new,shi2005extended}, deterministic gradient-based methods \cite{franceschi2017forward,franceschi2018bilevel,grazzi2020iteration,hu2023improved,maclaurin2015gradient,
pedregosa2016hyperparameter,rajeswaran2019meta}, and stochastic gradient-based methods \cite{chen2022single,ghadimi2018approximation,guo2021randomized,hong2023two,huang2021biadam,
huang2022efficiently,ji2021bilevel,khanduri2021near,kwon2023fully,li2022fully,yang2021provably} 
were proposed for solving  \eqref{BLO} with $g\equiv 0$, $\tg\equiv 0$,  $f$,  $\tf$ being smooth, and $\tf$ being \emph{strongly convex} with respect to $y$. For a similar case as this but with $\tf$ being \emph{convex} with respect to $y$, a zeroth-order method was recently proposed in \cite{chen2023bilevel}, and also numerical methods were developed in \cite{li2023novel,liu2022bome,sow2022constrained} by solving \eqref{BLO} as a single or sequential smooth constrained optimization problems. Besides, when all the functions in \eqref{BLO} are smooth and $\tf,\,\tg$ are \emph{convex} with respect to $y$, gradient-type methods were proposed by solving  a mathematical program with equilibrium constraints resulting from replacing the lower-level optimization problem of \eqref{BLO} by its first-order optimality conditions (e.g., see \cite{allende2013solving,luo1996mathematical,outrata2013nonsmooth}).  Recently, difference-of-convex (DC) algorithms were developed in \cite{ye2022difference} for solving \eqref{BLO} with $g\equiv 0$, $f$ being a DC function, and $\tf$, $\tg$ being convex functions. In addition, a double penalty method \cite{ishizuka1992double} was proposed for \eqref{BLO}, which solves a sequence of bilevel optimization problems of the form 
\beq\label{BLO-dp}
\begin{array}{rl}
\min\limits_{x,y} & f(x,y)+\rho_k\Psi(x,y)\\
\mbox{s.t.}&  y\in\argmin\limits_z \tf(x,z)+\rho_k\tilde\Psi(x,z),                
\end{array}
\eeq
 where $\{\rho_k\}$ is a sequence of penalty parameters, and $\Psi$ and $\tilde \Psi$ are a penalty function associated with the sets $\{(x,y)|g(x,y)\leq 0\}$ and $\{(x,z)|\tilde g(x,z)\leq 0\}$, respectively. Though problem \eqref{BLO-dp} appears to be simpler than \eqref{BLO}, there is no method available for finding an approximate solution of \eqref{BLO-dp} in general.  Consequently, the double penalty method \cite{ishizuka1992double} is typically not implementable. More discussion on algorithmic development for bilevel optimization can be found in   \cite{bard2013practical,colson2007overview,dempe2020bilevel,liu2021investigating,
sinha2017review,vicente1994bilevel}) and the references therein.

It has long been known that the notorious challenge of bilevel optimization \eqref{BLO} mainly comes from the lower level part, which requires that the variable $y$ be a solution of another optimization problem. Due to this, for the sake of simplicity, we only consider a subclass of bilevel optimization with the constraint $g(x,y)\leq 0$ being excluded, namely, 
  \begin{equation}\label{BLO-1}
\begin{array}{rl}
\min\limits_{x,y} & f(x,y)\\
\mbox{s.t.}&  y\in\argmin\limits_z\{\tf(x,z)|\tg(x,z)\leq0\}.                
\end{array}
\end{equation}  
Nevertheless, the results in this paper can be possibly extended to problem \eqref{BLO}.

The main goal of this paper is to develop an implementable first-order method with complexity guarantees for solving problem \eqref{BLO-1}. Our key insights for this development are: (i) problem \eqref{BLO-1} can be approximately solved as a structured minimax problem  that results from a novel penalty approach; (ii) the resulting structured minimax problem  can be suitably solved by a first-order method proposed in \cite[Algorithm 2]{lu2023first}. As a result, these lead to development of a novel first-order penalty method for solving \eqref{BLO-1}, which enjoys the following appealing features. 
\bi
\item It uses only the first-order information of the problem. Specifically, its fundamental operations consist only of gradient evaluation of $\tg$ and the smooth component of $f$ and $\tf$ and also proximal operator evaluation of the nonsmooth component of $f$ and $\tf$. Thus, it is  suitable for solving large-scale problems (see Sections \ref{unconstr-BLO} and \ref{constr-BLO}).
\item It has theoretical guarantees on operation complexity, which is measured by the aforementioned fundamental operations, for finding an $\varepsilon$-KKT solution of \eqref{BLO-1}. Specifically, when $\tg\equiv 0$, it enjoys an operation complexity of $\cO(\varepsilon^{-4}\log \varepsilon^{-1})$. Otherwise, it enjoys an operation complexity of $\cO(\varepsilon^{-7}\log \varepsilon^{-1})$ 
(see Theorems \ref{unc-complexity} and \ref{complexity}).
\item It is applicable to a broader class of problems than existing methods. For example, it can be applied to \eqref{BLO-1} with $f$, $\tf$ being nonsmooth and $\tf$, $\tg$ being nonconvex with respect to $x$, which is however not suitable for existing methods.
\ei
To the best of our knowledge, this paper is the first work to demonstrate that bilevel optimization can be approximately solved as minimax optimization, and moreover, it provides the first implementable method with complexity guarantees for the sophisticated bilevel optimization problem \eqref{BLO-1}.

The rest of this paper is organized as follows. In Subsection \ref{notation} we introduce some notation and terminology. 
In Sections~\ref{unconstr-BLO} and \ref{constr-BLO}, we propose first-order penalty methods for unconstrained and constrained bilevel optimization and study their complexity, respectively. Preliminary numerical results and the proofs of the main results are respectively presented in Sections \ref{sec:exp} and \ref{sec:proof}. Finally, we make some concluding remarks in Section \ref{sec:conclude}.

\subsection{Notation and terminology}  \label{notation}
The following notation will be used throughout this paper. Let $\bR^n$ denote the Euclidean space of dimension $n$ and $\bR^n_+$ denote the nonnegative orthant in $\bR^n$. The standard inner product and Euclidean norm are  respectively denoted by $\langle\cdot,\cdot\rangle$ and $\|\cdot\|$, unless stated otherwise.  For any $v\in\bR^n$, let $v_+$ denote the nonnegative part of $v$, that is, $(v_+)_i=\max\{v_i,0\}$ for all $i$. For any two vectors $u$ and $v$, $(u;v)$ denotes the vector resulting from stacking $v$ under $u$. Given a point $x$ and a closed set $S$ in $\bR^n$, let $\dist(x,S)=\min_{x'\in S} \|x'-x\|$ and $\cI_S$ denote the indicator function associated with $S$.

A function or mapping $\phi$ is said to be \emph{$L_{\phi}$-Lipschitz continuous} on a set $S$ if $\|\phi(x)-\phi(x')\| \leq L_{\phi} \|x-x'\|$ for all $x,x'\in S$. In addition, it is said to be \emph{$L_{\nabla\phi}$-smooth} on $S$ if $\|\nabla\phi(x)-\nabla\phi(x')\| \leq L_{\nabla\phi} \|x-x'\|$ for all $x,x'\in S$.\footnote{When $\phi$ is a mapping, the norm used in $\|\nabla\phi(x)-\nabla\phi(x')\|$ is the Frobenius norm.} For a closed convex function $p:\bR^n\to \bR\cup\{\infty\}$,\footnote{For convenience, $\infty$ stands for $+\infty$.} the \emph{proximal operator} associated with $p$ is denoted by  
$\prox_p$,  that is,
\[
\prox_p(x) = \argmin_{x'\in\bR^n} \left\{ \frac{1}{2}\|x' - x\|^2 + p(x') \right\} \quad \forall x \in \bR^n.
\]
Given that evaluation of $\prox_{\gamma p}(x)$ is often as cheap as $\prox_p(x)$, we count the evaluation of $\prox_{\gamma p}(x)$ as one evaluation of proximal operator of $p$ for any $\gamma>0$ and $x\in\bR^n$. 

For a lower semicontinuous function $\phi:\bR^n\to \bR\cup\{\infty\}$, its \emph{domain} is the set $\dom\, \phi := \{x| \phi(x)<\infty\}$. The \emph{upper subderivative} of $\phi$ at $x\in \dom\, \phi$ in a direction $d\in\bR^n$ is defined by
\[
\phi'(x;d) = \limsup\limits_{x' \stackrel{\phi}{\to} x,\, t \downarrow 0} \inf_{d' \to d} \frac{\phi(x'+td')-\phi(x')}{t},
\] 
where $t\downarrow 0$ means both $t > 0$ and $t\to 0$, and $x' \stackrel{\phi}{\to} x$ means both $x' \to x$ and $\phi(x')\to \phi(x)$. The \emph{subdifferential} of $\phi$ at $x\in \dom\, \phi$ is the set 
\[
\partial \phi(x) = \{s\in\bR^n\big| s^T d \leq \phi'(x; d) \ \ \forall d\in\bR^n\}.
\]
We use $\partial_{x_i} \phi(x)$ to denote the subdifferential with respect to $x_i$.  
In addition, for an upper semicontinuous function $\phi$, its subdifferential is defined as $\partial \phi=-\partial (-\phi)$. If $\phi$ is locally Lipschitz continuous, the above definition of subdifferential coincides with the Clarke subdifferential. Besides, if $\phi$ is convex, it coincides with the ordinary subdifferential for convex functions. Also, if $\phi$ is continuously differentiable at $x$ , we simply have $\partial \phi(x) = \{\nabla \phi(x)\}$, where $\nabla \phi(x)$ is the gradient of $\phi$ at $x$. In addition, it is not hard to verify that $\partial (\phi_1+\phi_2)(x)=\nabla \phi_1(x)+\partial \phi_2(x)$ if $\phi_1$ is continuously differentiable at $x$ and $\phi_2$ is lower or upper semicontinuous at $x$. See \cite{clarke1990optimization,ward1987nonsmooth} for more details.

%
%

Finally, we introduce two types of approximate solutions for a general minimax problem
\begin{equation}\label{eg}
\Psi^* =\min_{x}\max_{y}\Psi(x,y),
\end{equation}
where $\Psi(\cdot,y): \bR^n \to \bR \cup\{\infty\}$ is a lower semicontinuous function, $\Psi(x,\cdot): \bR^m \to \bR \cup\{-\infty\}$ is an upper semicontinuous function, and $\Psi^*$ is finite. 

\begin{defi} \label{def1}
 A point $(\xe,\ye)$ is called an $\epsilon$-optimal solution of the minimax problem \eqref{eg} if
\begin{equation*}
 \max_{y}\Psi(x_\epsilon,y)-\Psi(\xe,\ye)\leq\epsilon,\quad\Psi(\xe,\ye)-\Psi^*\leq\epsilon.
\end{equation*}
\end{defi}

\begin{defi}  \label{def2}
 A point $(x,y)$ is called a stationary point of the minimax problem \eqref{eg} if 
\[
0 \in \partial_x\Psi(x,y), \quad 0\in\partial_y\Psi(x,y).
\]
In addition, for any $\epsilon>0$, a point $(\xe,\ye)$ is called an $\epsilon$-stationary point of the minimax problem \eqref{eg} if
\begin{equation*}
\dist\left(0,\partial_x\Psi(\xe,\ye)\right)\leq\epsilon,\quad\dist\left(0,\partial_y\Psi(\xe,\ye)\right)\leq\epsilon.
\end{equation*}
\end{defi}

\section{Unconstrained bilevel optimization} \label{unconstr-BLO}
In this section, we consider an unconstrained bilevel optimization problem\footnote{For convenience, problem \eqref{unc-prob} is referred to as an unconstrained bilevel optimization problem since its lower level part does not have an explicit constraint. Strictly speaking, it can be a constrained bilevel optimization problem. For example, when part of $f$ and/or $\tf$ is the indicator function of a closed convex set, \eqref{unc-prob} is essentially a constrained bilevel optimization problem.}
\begin{equation}\label{unc-prob}
\begin{array}{rl}
f^*=\min & f(x,y)\\
\mbox{s.t.}& y\in\argmin\limits_{z}\tf(x,z).
\end{array}
\end{equation}
Assume that problem \eqref{unc-prob} has at least one optimal solution. In addition, $f$ and $\tf$ satisfy the following assumptions.
\begin{assumption}\label{a1}
\begin{enumerate}[label=(\roman*)]
\item $f(x,y)=f_1(x,y)+f_2(x)$ and $\tf(x,y)=\tf_1(x,y)+\tf_2(y)$ are continuous on $\mcX \times \mcY$, where $f_2: \bR^n \to \bR\cup \{\infty\}$  and $\tf_2:\bR^m \to \bR\cup \{\infty\}$ are proper closed convex functions, $\tf_1(x,\cdot)$ is convex for any given $x\in\mcX$, and $f_1$, $\tf_1$ are respectively $L_{\nabla f_1}$- and $L_{\nabla \tf_1}$-smooth on $\mcX \times \mcY$ with $\mcX:=\dom\,f_2$ and $\mcY:=\dom\,\tf_2$.
\item The proximal operator associated with $f_2$ and $\tf_2$ can be exactly evaluated. 
\item The sets $\mcX$ and $\mcY$ (namely, $\dom\,f_2$ and $\dom\,\tf_2$) are compact.
\end{enumerate}
\end{assumption}

For notational convenience, we define
\begin{align}
&D_\bfx\coloneqq \max\{\|u-v\|\big|u,v\in\mcX\},\quad D_\bfy\coloneqq\max\{\|u-v\|\big|u,v\in\mcY\},\label{DxDy}\\
&\tf_{\rm hi}:=\max\{\tf(x,y)|(x,y)\in\mcX\times\mcY\},\quad\tf_{\rm low}:=\min\{\tf(x,y)|(x,y)\in\mcX\times\mcY\}, \label{tfbnd} \\
&f_{\rm low}:=\min\{f(x,y)|(x,y)\in \mcX \times \mcY\}. \label{lower-bnd}
\end{align}
By Assumption~\ref{a1}, one can observe that $D_\bfx$, $D_\bfy$, $\tf_{\rm hi}$, $\tf_{\rm low}$ and $f_{\rm low}$ are finite.

The goal of this section is to propose first-order penalty methods for solving problem \eqref{unc-prob}. To this end, we first observe that problem \eqref{unc-prob} can be viewed as 
\beq \label{unc-prob-ref}
\min_{x,y}\{f(x,y)|\tf(x,y)\leq\min\limits_{z}\tf(x,z)\}.
\eeq
 Notice that $\tf(x,y)-\min_{z}\tf(x,z) \geq 0$ for all $x,y$. Consequently, a natural \emph{penalty problem} associated with \eqref{unc-prob-ref} 
 is 
\beq \label{exact-penalty}
\min_{x,y}f(x,y)+\rho(\tf(x,y)-\min_{z}\tf(x,z)),
\eeq
where $\rho>0$ is a penalty parameter. We further observe that  \eqref{exact-penalty}
is equivalent to the \emph{minimax problem}
\begin{equation}\label{unc-mmax}
\min_{x,y}\max_{z}P_\rho(x,y,z),\quad\mbox{where}\quad P_\rho(x,y,z)\coloneqq f(x,y)+\rho(\tf(x,y)-\tf(x,z)).
\end{equation}
In view of Assumption \ref{a1}(i), $P_\rho$ can be rewritten as
\beq\label{P-rho}
P_\rho(x,y,z) = \big(f_1(x,y)+\rho\tf_1(x,y)-\rho\tf_1(x,z)\big)+\big(f_2(x) +\rho\tf_2(y)-\rho\tf_2(z)\big).
\eeq
By this and Assumption \ref{a1}, one can observe that $P_\rho$ enjoys the following nice properties.
\bi
\item $P_\rho$ is the sum of smooth function $f_1(x,y)+\rho\tf_1(x,y)-\rho\tf_1(x,z)$ with Lipschitz continuous gradient and possibly nonsmooth function $f_2(x) +\rho\tf_2(y)-\rho\tf_2(z)$ with exactly computable proximal operator.
\item $P_\rho$ is nonconvex in $(x,y)$ but concave in $z$.
\ei
Thanks to this nice structure of $P_\rho$, an approximate stationary point of the minimax problem \eqref{unc-mmax} can be found by a first-order method proposed in \cite[Algorithm 2]{lu2023first} (see Algorithm \ref{mmax-alg2} in Appendix A).

Based on the above observations, we are now ready to propose penalty methods for the unconstrained bilevel optimization problem \eqref{unc-prob} by solving either a sequence of minimax problems or a single minimax problem in the form of \eqref{unc-mmax}. Specifically, we first propose an \emph{ideal} penalty method for \eqref{unc-prob} by solving a sequence of minimax problems (see Algorithm \ref{alg1}). Then we propose a \emph{practical} penalty method for \eqref{unc-prob} by finding an approximate stationary point of a single minimax problem (see Algorithm \ref{alg2}).

\begin{algorithm}[H]
\caption{An ideal penalty method for problem~\eqref{unc-prob}}
\label{alg1}
\begin{algorithmic}[1]
\REQUIRE positive sequences $\{\rho_k\}$ and $\{\epsilon_k\}$ with $\lim_{k\to\infty}(\rho_k,\epsilon_k)=(\infty,0)$.
\FOR{$k=0,1,2,\ldots$}
\STATE Find an $\epsilon_k$-optimal solution $(x^k,y^k,z^k)$ of problem \eqref{unc-mmax} with $\rho=\rho_k$.
\ENDFOR
\end{algorithmic}
\end{algorithm}

The following theorem states a convergence result of Algorithm \ref{alg1}, whose proof is deferred to Section \ref{sec:proof3}.

\begin{thm}[{\bf Convergence of Algorithm \ref{alg1}}]\label{c1}
Suppose that Assumption \ref{a1} holds and that $\{(x^k,y^k,z^k)\}$ is generated by Algorithm \ref{alg1}. Then any accumulation point of $\{(x^k,y^k)\}$ is an optimal solution of problem \eqref{unc-prob}.
\end{thm}

Notice that \eqref{unc-mmax} is a \emph{nonconvex}-concave minimax problem. It is typically hard to find an $\epsilon$-optimal solution of \eqref{unc-mmax} for an arbitrary $\epsilon>0$. Consequently, Algorithm \ref{alg1} is \emph{not implementable} in general. We next propose a \emph{practical} penalty method for  problem \eqref{unc-prob} by 
applying Algorithm~\ref{mmax-alg2} (see Appendix A) to find an approximate stationary point of a single minimax problem \eqref{unc-mmax} with a suitable choice of $\rho$.

\begin{algorithm}[H]
\caption{A practical penalty method for problem~\eqref{unc-prob}}
\label{alg2}
\begin{algorithmic}[1]
\REQUIRE  $\varepsilon\in(0,1/4]$, $\rho=\varepsilon^{-1}$, $(x^0,y^0)\in\mcX\times\mcY$  with $\tf(x^0,y^0)\leq\min_y\tf(x^0,y)+\varepsilon$.
\STATE
Call Algorithm~\ref{mmax-alg2} in Appendix A with $\epsilon\leftarrow\varepsilon$, $\epsilon_0\leftarrow\varepsilon^{3/2}$, $\hat x^0\leftarrow(x^0,y^0)$, $\hat y^0\leftarrow y^0$, and $L_{\nabla \bh}\leftarrow L_{\nabla f_1}+2\varepsilon^{-1} L_{\nabla \tf_1}$
to find an $\varepsilon$-stationary point $(\xe,\ye,\ze)$ of problem \eqref{unc-mmax} with $\rho=\varepsilon^{-1}$. 
\STATE \noindent\textbf{Output}: $(\xe,\ye)$.
\end{algorithmic}
\end{algorithm}

\begin{rem} \label{alg4-remark}
(i) The initial point $(x^0,y^0)$ of Algorithm \ref{alg2} can be found by an additional procedure. Indeed, one can first choose any $x^0\in\mcX$ and then apply accelerated proximal gradient method \cite{nesterov2013gradient} to the problem $\min_y\tf(x^0,y)$ for finding $y^0\in\mcY$ such that $\tf(x^0,y^0)\leq\min_y\tf(x^0,y)+\varepsilon$; 
(ii) As seen from Theorem~\ref{mmax-thm} (see Appendix A), an $\epsilon$-stationary point of \eqref{unc-mmax} can be successfully found in step 1 of Algorithm~\ref{alg2} by applying Algorithm~\ref{mmax-alg2} to \eqref{unc-mmax};
(iii) For the sake of simplicity, a single subproblem of the form \eqref{unc-mmax} with static penalty and tolerance parameters is solved in Algorithm \ref{alg2}. Nevertheless, Algorithm \ref{alg2} can be modified into a perhaps practically more efficient algorithm by solving a sequence of subproblems of the form \eqref{unc-mmax} with dynamic penalty and tolerance parameters instead.
\end{rem}

In order to characterize the approximate solution found by Algorithm~\ref{alg2}, we next introduce a notion of $\varepsilon$-KKT solution of problem \eqref{unc-prob}.


Recall that problem \eqref{unc-prob} can be viewed as problem \eqref{unc-prob-ref}, which is a constrained optimization problem. In the spirit of classical constrained optimization, one would naturally be interested in a KKT solution $(x,y)$ of \eqref{unc-prob-ref} or equivalently \eqref{unc-prob}, namely, $(x,y)$ satisfies $\tf(x,y) \leq \min_z\tf(x,z)$ and moreover $(x,y)$ is a stationary point of the problem
\beq \label{min-prob}
\min_{x',y'}f(x',y')+\rho\big(\tf(x',y')-\min_{z'}\tf(x',z')\big)
\eeq 
for some $\rho\geq 0$.\footnote{The relation $\tf(x,y) \leq \min_z\tf(x,z)$ implies that $\tf(x,y)=\min_z\tf(x,z)$ and hence the complementary slackness condition 
$\rho(\tf(x,y)-\min_z\tf(x,z))=0$ holds.} Yet, due to the sophisticated problem structure, characterizing a stationary point of  \eqref{min-prob} is generally difficult. On another hand, notice that problem \eqref{min-prob} is equivalent to the minimax problem
\[
\min_{x',y'}\max_{z'} f(x',y')+\rho(\tf(x',y')-\tf(x',z')),
\]
whose stationary point $(x,y,z)$, according to Definition \ref{def2} and Assumption \ref{a1},  satisfies 
\begin{align}
0\in\partial f(x,y)+\rho\partial\tf(x,y) - (\rho\nabla_x\tf(x,z);0), \quad 0\in\rho\partial_z\tf(x,z).
 \label{kkt-unc}
\end{align}
Based on this observation, we are instead interested in a (weak) KKT solution of problem \eqref{unc-prob} and its inexact counterpart that are defined below.

\begin{defi}
The pair $(x,y)$ is said to be a KKT solution of problem \eqref{unc-prob} if there exists $(z,\rho)\in\bR^m\times\bR_+$ such that \eqref{kkt-unc} and $\tf(x,y) \leq \min_{z'}\tf(x,z')$ hold. In addition, for any $\varepsilon>0$, $(x,y)$ is said to be an $\varepsilon$-KKT solution of problem \eqref{unc-prob} if there exists $(z,\rho)\in\bR^m\times\bR_+$ such that
\begin{align*}
&\dist\Big(0,\partial f(x,y)+\rho\partial\tf(x,y) - (\rho\nabla_x\tf(x,z); 0)\Big)\leq\varepsilon, \quad \dist\big(0,\rho\partial_z\tf(x,z)\big)\leq\varepsilon,\\
&\tf(x,y)-\min_{z'}\tf(x,z')\leq\varepsilon.
\end{align*}
\end{defi}

Recently, a hypergradient-based stationary point has been considered in the literature (e.g., \cite{ghadimi2018approximation,pedregosa2016hyperparameter}) for problem \eqref{unc-prob} under the assumption that $f$ and $\tf$ are twice continuously differentiable in $\bR^n\times\bR^m$ and $\tf(x,\cdot)$ is strongly convex for any $x\in\bR^n$. Under this assumption,  the hyper-objective function $\Phi$ of \eqref{unc-prob}, defined as 
\beq \label{Phi}
 \Phi(x):=f(x,y^*(x)), \ \mbox{where} \  y^*(x)=\argmin\limits_z\tf(x,z),
\eeq
is continuously differentiable. Moreover,  following from \cite[Equation (2.8)]{ghadimi2018approximation}, the hypergradient of \eqref{unc-prob}, i.e., the gradient of $\Phi$,  is given by
\begin{equation} \label{hypergrad}
\nabla \Phi(x)=\nabla_x f(x,y^*(x))-\nabla_{xy}^2\tf(x,y^*(x))[\nabla_{yy}^2\tf(x,y^*(x))]^{-1}\nabla_yf(x,y^*(x)) \quad \forall x\in\bR^n.
\end{equation}
In addition, it is not hard to observe that problem \eqref{unc-prob} is equivalent to 
\begin{equation}\label{sprob}
\min\limits_{x} \, \Phi(x).
\end{equation}
In view of this, hypergradient based stationary point and its approximate counterpart are introduced in the literature (e.g., \cite{ghadimi2018approximation,pedregosa2016hyperparameter}) for problem \eqref{unc-prob}, based on the classical stationary point and  its approximate counterpart of problem \eqref{sprob}. More specifically, $x\in \bR^n$ is called a {\it hypergradient-based stationary point} of problem \eqref{unc-prob} if $\nabla \Phi(x)=0$, and it is called an 
{\it $\varepsilon$-hypergradient-based stationary point} of \eqref{unc-prob} if $\|\nabla \Phi(x)\|\leq \varepsilon$ for any $\varepsilon>0$.

We now study the relationship between an $\varepsilon$-KKT solution and an approximate hypergradient based stationary point of problem \eqref{unc-prob}. Specifically, under some suitable assumptions, the following theorem shows that if $(x,y)$ is an $\varepsilon$-KKT solution of problem \eqref{unc-prob}, then $x$ is an $\cO(\sqrt{\varepsilon})$- or $\cO(\varepsilon)$-hypergradient-based stationary point of it.  The proof of this theorem is deferred to Subsection \ref{sec:proof3}.


\begin{thm} \label{hypergrad-thm}
Let $\varepsilon_0, \rho_0>0$ be given and $\Omega \subset \bR^n$ be a nonempty compact set.  Assume that $f$ and $\tf$ are continuously differentiable and twice continuously differentiable in $\bR^n\times\bR^m$ respectively,  $\tf(x',\cdot)$ is strongly convex with modulus $\sigma>0$ for all $x'$ 
in an open set  $\calN$ containing $\Omega$, $\nabla f(x',\cdot)$ is $L_1$-Lipschitz continuous for all $x'\in\Omega$, and that $\nabla^2 \tf(x',\cdot)$ is $L_2$-Lipschitz continuous for all $x' \in \Omega$. 
Suppose that $(x,y)\in \Omega \times \bR^m$ is an $\varepsilon$-KKT solution of problem \eqref{unc-prob} with its associated $\rho\geq \rho_0$ for some $0<\varepsilon \leq \varepsilon_0$. 
Let $y^*(x')$ be defined in \eqref{Phi} and
\begin{align}
&\bar C=\max\big\{\|\nabla_y f(x',y')\|:x'\in\Omega,\|y'-y^*(x')\|\leq\sqrt{2\sigma^{-1}\varepsilon}\big\}, \label{def-Mf} \\
&\theta=\min\big\{(\rho\sigma)^{-1}(\varepsilon+\bar C),\ \sqrt{2\sigma^{-1}\varepsilon}\big\}, \quad C=\max_{x'\in\Omega}\|\nabla_{xy}^2\tf(x',y^*(x'))[\nabla_{yy}^2\tf(x',y^*(x'))]^{-1}\|. \label{C} 
\end{align}
Then we have 
\begin{align} 
\|\nabla\Phi(x)\| & \leq(2C+1)\varepsilon+(C+1)\Big(L_1\theta+\frac{ L_2\rho\theta^2}{2}+\frac{L_2\varepsilon^2}{2\rho\sigma^2}\Big) \label{hypergrad-bnd1} \\
& \leq (2C+1)\varepsilon+(C+1)\Big(L_1\sqrt{2\sigma^{-1}}+\frac{L_2\sigma^{-3/2}(\varepsilon_0+\bar C)}{\sqrt{2}}+\frac{\varepsilon^{3/2}_0L_2}{2\rho_0\sigma^2}\Big)\sqrt{\varepsilon}.  \label{hypergrad-bnd2}
\end{align}
\end{thm}

\begin{rem}
Based on the assumptions in Theorem \ref{hypergrad-thm}, it is not hard to observe that $\bar C$, $C$ and $\theta$ are finite. It then follows from \eqref{hypergrad-bnd2} that $\nabla \Phi(x)=\cO(\sqrt{\varepsilon})$ holds in general.  Nevertheless, this result is improved to $\nabla \Phi(x)=\cO(\varepsilon)$ when $\rho$ is at least order of $\varepsilon^{-1}$, i.e., $\rho \geq c \varepsilon^{-1}$ for some constant $c>0$ independent on $\varepsilon$, which can be observed from \eqref{hypergrad-bnd1}. Consequently, under the assumptions stated in Theorem \ref{hypergrad-thm}, if $(x,y)$ is an $\varepsilon$-KKT solution of problem \eqref{unc-prob}, then $x$ is an $\cO(\sqrt{\varepsilon})$- or $\cO(\varepsilon)$-hypergradient-based stationary point of it.  
\end{rem}

We next present a theorem regarding \emph{operation complexity} of Algorithm~\ref{alg2}, measured by the amount of evaluations of $\nabla f_1$, $\nabla\tf_1$ and proximal operator of $f_2$ and $\tf_2$, for finding an $\cO(\varepsilon)$-KKT solution of  \eqref{unc-prob}, whose proof is deferred to Subsection \ref{sec:proof3}.

\begin{thm}[{\bf Complexity of Algorithm \ref{alg2}}]\label{unc-complexity}
Suppose that Assumption~\ref{a1} holds. Let $f^*$, $f$, $\tf$, $D_\bfx$, $D_\bfy$, $\tf_{\rm hi}$, $\tf_{\rm low}$ and $f_{\rm low}$ be defined in \eqref{unc-prob}, \eqref{DxDy}, \eqref{tfbnd} and \eqref{lower-bnd}, $L_{\nabla f_1}$ and $L_{\nabla \tf_1}$ be  given in Assumption \ref{a1}, $\varepsilon$, $\rho$, $x^0$, $y^0$ and $\ze$ be given in Algorithm~\ref{alg2}, and 
\begin{align}
&\hL=L_{\nabla f_1}+2\varepsilon^{-1}L_{\nabla \tf_1},\ \hat\alpha=\min\left\{1,\sqrt{4\varepsilon/(D_\bfy\hL)}\right\},\label{hL}\\
&\hat \Cr=(2+\hat\alpha^{-1})(D_\bfx^2+D_\bfy^2)\hL+\max\left\{\varepsilon/D_\bfy,\hat\alpha\hL/4\right\}D_\bfy^2,\notag\\
&\widehat C=\frac{4\max\left\{\frac{1}{2\hL},\min\left\{\frac{D_\bfy}{\varepsilon},\frac{4}{\hat \alpha \hL}\right\}\right\}\left[\hat\Cr +2\hat \alpha^{-1}(f^*-f_{\rm low}+\varepsilon^{-1}(\tf_{\rm hi}-\tf_{\rm low})+\varepsilon D_\bfy/4+\hL (D_\bfx^2+D_\bfy^2))\right]}{\left[(3\hL+\varepsilon/(2D_\bfy))^2/\min\{\hL,\varepsilon/(2D_\bfy)\}+ 3\hL+\varepsilon/(2D_\bfy)\right]^{-2}\varepsilon^3},\notag\\
&\widehat K=\left\lceil16(1+f(x^0,y^0)-f_{\rm low}+\varepsilon D_\bfy/4)\hL\varepsilon^{-2}+32(1+4D_\bfy^2\hL^2\varepsilon^{-2})\varepsilon-1\right\rceil_+,\notag\\
&\widehat N=\left(\left\lceil96\sqrt{2}(1+(24\hL+4\varepsilon/D_\bfy)\hL^{-1})\right\rceil+2\right)\max\left\{2,\sqrt{D_\bfy\hL\varepsilon^{-1}}\right\}\notag\\
&\ \ \ \ \ \times((\widehat K+1)(\log \widehat C)_++\widehat K+1+2\widehat K\log(\widehat K+1)).\notag
\end{align}
Then Algorithm~\ref{alg2} outputs an approximate solution $(\xe,\ye)$ of \eqref{unc-prob} satisfying
\begin{align}
&\dist\left(0,\partial f(\xe,\ye)+\rho\partial\tf(\xe,\ye)-(\rho\nabla_x\tf(\xe,\ze);0)\right)\leq\varepsilon,\quad\dist\left(0,\rho\partial_z\tf(\xe,\ze)\right)\leq\varepsilon,\label{unc-gap1}\\
&\tf(\xe,\ye)\leq\min_{z}\tf(\xe,z)+\varepsilon\left(1+f(x^0,y^0)-f_{\rm low}+2\varepsilon^3(\hL^{-1}+4D_\bfy^2\hL\varepsilon^{-2})+D_\bfy\varepsilon/4\right),\label{unc-gap2}
\end{align}
after at most $\widehat N$ evaluations of $\nabla f_1$, $\nabla\tf_1$ and proximal operator of $f_2$ and $\tf_2$, respectively.
\end{thm}

\begin{rem}
One can observe from Theorem \ref{unc-complexity} that  $\hL=\cO(\varepsilon^{-1})$, $\hat\alpha=\cO(\varepsilon)$, $\hat\Cr=\cO(\varepsilon^{-2})$, $\widehat C=\cO(\varepsilon^{-11})$, $\widehat K=\cO(\varepsilon^{-3})$, and $\widehat N=\cO(\varepsilon^{-4}\log \varepsilon^{-1})$. As a result, Algorithm~\ref{alg2} enjoys an operation complexity  of $\cO(\varepsilon^{-4}\log \varepsilon^{-1})$, measured by the amount of evaluations of $\nabla f_1$, $\nabla\tf_1$ and proximal operator of $f_2$ and $\tf_2$, for finding an $\cO(\varepsilon)$-KKT solution $(\xe,\ye)$ of \eqref{unc-prob} satisfying
\begin{align*}
&\dist\left(0,\partial f(\xe,\ye)+\rho\partial\tf(\xe,\ye)-(\rho\nabla_x\tf(\xe,\ze);0)\right)\leq\varepsilon,\quad\dist\left(0,\rho\partial_z\tf(\xe,\ze)\right)\leq\varepsilon,\\
&\tf(\xe,\ye)-\min_z\tf(\xe,z)=\cO(\varepsilon),
\end{align*}
where $\ze$ is given in Algorithm \ref{alg2} and $\rho=\varepsilon^{-1}$.
\end{rem}


\section{Constrained bilevel optimization} \label{constr-BLO}
In this section, we consider a constrained bilevel optimization problem\footnote{For convenience, problem \eqref{prob} is referred to as a  constrained bilevel optimization problem since its lower level part has at least one explicit constraint.}
\begin{equation}\label{prob}
\begin{array}{rl}
f^*=\min & f(x,y)\\
\mbox{s.t.}& y\in\argmin\limits_z\{\tf(x,z)|\tg(x,z)\leq0\},
\end{array}
\end{equation}
where $f$ and $\tf$ satisfy Assumption~\ref{a1}. 
Recall from Assumption~\ref{a1} that $\mcX=\dom\,f_2$ and $\mcY=\dom\,\tf_2$. 
We now make some additional assumptions for problem \eqref{prob}.

\begin{assumption}\label{a2}
\begin{enumerate}[label=(\roman*)]
\item $f$ and $\tf$ are $L_f$- and $L_\tf$-Lipschitz continuous on $\mcX \times \mcY$, respectively. 
\item $\tg:\bR^n\times\bR^m\to\bR^l$ is $L_{\nabla \tg}$-smooth and $L_\tg$-Lipschitz continuous on $\mcX\times\mcY$.
\item $\tg_i(x,\cdot)$ is convex and there exists $\hat z_x\in\mcY$ for each $x\in\mcX$ such that $\tg_i(x,\hat z_x)<0$ for all $i=1,2,\dots, l$ and $G:=\min\{-\tg_i(x,\hat z_x)|x\in\mcX,\ i=1,\dots, l\}>0$.\footnote{The latter part of this assumption can be weakened to the one that the pointwise Slater's condition holds for the lower level part of \eqref{prob}, that is, there exists $\hat z_x\in\mcY$ such that $\tg(x,\hat z_x)<0$ for each $x\in\mcX$. Indeed, if $G>0$, Assumption \ref{a2}(iii) clearly holds. Otherwise, one can solve the perturbed counterpart of \eqref{prob} with $\tg(x,z)$ being replaced by $\tg(x,z)-\epsilon$ for some suitable $\epsilon>0$ instead, which satisfies Assumption \ref{a2}(iii).}
\end{enumerate}
\end{assumption}

For notational convenience, we define
\begin{align}
& \tf^*(x)\coloneqq\min\limits_z\{\tf(x,z)|\tg(x,z)\leq0\},\label{tfstarx}\\
& \tf^*_{\rm hi} :=\sup\{\tf^*(x)|x\in\mcX\}, \label{def-tFx}\\
&\tg_{\rm hi}:= \max\{\|\tg(x,y)\|\big| (x,y)\in\mcX\times\mcY\}. \label{dg-bnd}
\end{align}
It then follows from Assumption \ref{a2}(ii) that 
\beq \label{tg-bnd}
\|\nabla \tg(x,y)\| \leq L_\tg \qquad \forall (x,y)\in\mcX\times\mcY.
\eeq
In addition, by Assumptions \ref{a1} and \ref{a2} and the compactness of $\mcX$ and $\mcY$, one can observe that $\tg_{\rm hi}$ and $G$ are finite. Besides, as will be shown in Lemma \ref{dual-bnd}(ii), $\tf^*_{\rm hi}$ is finite. 

The goal of this section is to propose first-order penalty methods for solving problem \eqref{prob}. To this end, let us first 
introduce a \emph{penalty function} for the lower level optimization problem $y\in\argmin\limits_z\{\tf(x,z)|\tg(x,z)\leq0\}$ of \eqref{prob}, which is given by
\begin{equation} \label{def-tFmu}
\tP_\mu(x,z)=\tf(x,z)+\mu\left\|[\tg(x,z)]_+\right\|^2
\end{equation}
for a penalty parameter $\mu>0$. Observe that problem \eqref{prob} can be approximately solved as the \emph{unconstrained bilevel optimization }problem 
\begin{equation}\label{def-Fmu}
f_\mu^*=\min_{x,y}\big\{f(x,y)|y\in\argmin\limits_{z}\tP_\mu(x,z)\big\}.
\end{equation}
Further, by the study in Section \ref{unconstr-BLO}, problem \eqref{def-Fmu} can be approximately solved as the \emph{penalty problem} 
\beq \label{penalty-prob}
\min_{x,y}f(x,y)+\rho\Big(\tP_\mu(x,y)-\min_{z}\tP_\mu(x,z)\Big)
\eeq 
for some suitable $\rho>0$. One can also observe that problem \eqref{penalty-prob} is equivalent to the \emph{minimax problem}
\begin{equation}\label{mmax}
\min_{x,y}\max_{z} P_{\rho,\mu}(x,y,z),\quad\mbox{where}\quad P_{\rho,\mu}(x,y,z)\coloneqq f(x,y)+\rho(\tP_\mu(x,y)-\tP_\mu(x,z)).
\end{equation}
In view of  \eqref{def-tFmu}, \eqref{mmax} and Assumption \ref{a1}(i), $P_{\rho,\mu}$ can be rewritten as 
\begin{align}
P_{\rho,\mu}(x,y,z) = & \  \big(f_1(x,y)+\rho\tf_1(x,y)+\rho\mu\left\|[\tg(x,y)]_+\right\|^2-\rho\tf_1(x,z)-\rho\mu\left\|[\tg(x,z)]_+\right\|^2\big)\notag\\ 
& \ +\big(f_2(x) +\rho\tf_2(y)-\rho\tf_2(z)\big).\label{Prhomu-ref}
\end{align}
By this and Assumptions \ref{a1} and \ref{a2}, one can observe that $P_{\rho,\mu}$ enjoys the following nice properties.
\bi
\item $P_{\rho,\mu}$ is the sum of smooth function $f_1(x,y)+\rho\tf_1(x,y)+\rho\mu\left\|[\tg(x,y)]_+\right\|^2-\rho\tf_1(x,z)-\rho\mu\left\|[\tg(x,z)]_+\right\|^2$ with Lipschitz continuous gradient and possibly nonsmooth function $f_2(x) +\rho\tf_2(y)-\rho\tf_2(z)$ with  exactly computable proximal operator;
\item $P_{\rho,\mu}$ is nonconvex in $(x,y)$ but concave in $z$.
\ei
Due to this nice structure of $P_{\rho,\mu}$, an approximate stationary point of the minimax problem \eqref{mmax} can be found by a first-order method proposed in \cite[Algorithm 2]{lu2023first} (see Algorithm \ref{mmax-alg2} in Appendix A).

Based on the above observations, we are now ready to propose penalty methods for the constrained bilevel optimization problem \eqref{prob} by solving a sequence of minimax problems or a single minimax problem of the form \eqref{mmax}. Specifically, we first propose an \emph{ideal} penalty method for \eqref{prob} by solving a sequence of minimax problems (see Algorithm \ref{alg3}). Then we propose a \emph{practical} penalty method for \eqref{prob} by finding an approximate stationary point of a single minimax problem (see Algorithm \ref{alg4}).

\begin{algorithm}[H]
\caption{An ideal penalty method for problem~\eqref{prob}}
\label{alg3}
\begin{algorithmic}[1]
\REQUIRE positive sequences $\{\rho_k\}$, $\{\mu_k\}$ and $\{\epsilon_k\}$ with $\lim_{k\to\infty}(\rho_k,\mu_k,\epsilon_k)=(\infty,\infty,0)$.
\FOR{$k=0,1,2,\ldots$}
\STATE Find an $\epsilon_k$-optimal solution $(x^k,y^k,z^k)$ of problem \eqref{mmax} with $(\rho,\mu)=(\rho_k,\mu_k)$.
\ENDFOR
\end{algorithmic}
\end{algorithm}


To study convergence of Algorithm \ref{alg3}, we make the following error bound assumption on the solution set of the lower level optimization problem of \eqref{prob}. This type of error bounds has been considered in the context of set-value mappings in the literature (e.g., see \cite{dontchev2009implicit}).

\begin{assumption}\label{errorbnd1}
There exist $\bar\theta>0$ and a non-decreasing function $\omega:\bR_+\to\bR_+$ with $\lim_{\theta \downarrow 0} \omega(\theta)=0$ such that $\dist(z,\cS_\theta(x))\leq \omega(\theta)$ for all $x\in\mcX$, $z\in \cS_0(x)$ and $\theta\in [0,\bar\theta]$, where
\begin{equation*}
\cS_\theta(x)\coloneqq\argmin_z\{\tf(x,z):\left\|[\tg(x,z)]_+\right\|\leq\theta\}.
\end{equation*}
\end{assumption}

We are now ready to state a convergence result of Algorithm \ref{alg3}, whose proof is deferred to Section \ref{sec:proof4}.

\begin{thm}[{\bf Convergence of Algorithm \ref{alg3}}]\label{c3}
Suppose that Assumptions \ref{a1}-\ref{errorbnd1} hold and that $\{(x^k,y^k,z^k)\}$ is generated by Algorithm~\ref{alg3}. Then any accumulation point of $\{(x^k,y^k)\}$ is an optimal solution of problem \eqref{prob}.
\end{thm}

Notice that \eqref{mmax} is a \emph{nonconvex}-concave minimax problem. It is typically hard to find an $\epsilon$-optimal solution of \eqref{mmax} for an arbitrary $\epsilon>0$. As a result, Algorithm \ref{alg3} is generally \emph{not implementable}. We next propose a \emph{practical} penalty method for problem \eqref{prob} 
by applying Algorithm~\ref{mmax-alg2} (see Appendix A) to find an approximate stationary point of \eqref{mmax} with a suitable choice of $\rho$ and $\mu$.

\begin{algorithm}[H]
\caption{A practical penalty method for problem~\eqref{prob}}
\label{alg4}
\begin{algorithmic}[1]
\REQUIRE $\varepsilon\in(0,1/4]$, $\rho=\varepsilon^{-1}$,  $\mu=\varepsilon^{-2}$, $(x^0,y^0)\in\mcX\times\mcY$ with $\tP_\mu(x^0,y^0)\leq\min_{y}\tP_\mu(x^0,y)+\varepsilon$. 
\STATE
Call Algorithm~\ref{mmax-alg2} in Appendix A with $\epsilon\leftarrow\varepsilon$, $\epsilon_0\leftarrow\varepsilon^{5/2}$, $\hat x^0\leftarrow(x^0,y^0)$, $\hat y^0\leftarrow y^0$, and $L_{\nabla \bh}\leftarrow L_{\nabla f_1}+2\rho L_{\nabla\tf_1}+4\rho\mu (\tg_{\rm hi}L_{\nabla \tg}+L_\tg^2)$
to find an $\varepsilon$-stationary point $(\xe,\ye,\ze)$ of problem \eqref{mmax} with $\rho=\varepsilon^{-1}$ and $\mu=\varepsilon^{-2}$.
\STATE \noindent\textbf{Output}: $(\xe,\ye)$.
\end{algorithmic}
\end{algorithm}

\begin{rem}
(i) The initial point $(x^0,y^0)$ of Algorithm \ref{alg4} can be found by a similar procedure as described in Remark \ref{alg4-remark} with $\tf$ being replaced by $\tP_\mu$; (ii) The choice of $\rho=\varepsilon^{-1}$ and $\mu=\varepsilon^{-2}$ is crucial in terms of the order of $\varepsilon$  for Algorithm \ref{alg4} to find an $\cO(\varepsilon)$-KKT solution of problem \eqref{prob} with a best operation complexity among all possible choices of $\rho$ and $\mu$, which can be observed from the proof of Theorem \ref{complexity}.  Intuitively, the lower level penalty parameter $\mu$ has to be larger than the upper level parameter $\rho$ in terms of the order of $\varepsilon$ so that the incurred error from the lower level constraint violation is not magnified (see \eqref{rel7} in Lemma \ref{t4}). (iii) As seen from Theorem~\ref{mmax-thm} (see Appendix A), an $\epsilon$-stationary point of \eqref{mmax} can be successfully found in step 1 of Algorithm~\ref{alg4} by applying Algorithm~\ref{mmax-alg2} to \eqref{mmax}; 
(iv) For the sake of simplicity, a single subproblem of the form \eqref{mmax} with static penalty and tolerance parameters is solved in Algorithm \ref{alg4}. Nevertheless, Algorithm \ref{alg4} can be modified into a perhaps practically more efficient algorithm by solving a sequence of subproblems of the form \eqref{mmax} with dynamic penalty and tolerance parameters instead.
\end{rem}



In order to characterize the approximate solution found by Algorithm~\ref{alg4}, we next introduce a notion of $\varepsilon$-KKT solution of problem \eqref{prob}.

By the definition of $\tf^*$ in \eqref{tfstarx}, problem \eqref{prob} can be viewed as 
\beq \label{prob-ref}
\min_{x,y}\{f(x,y)|\tf(x,y)\leq \tf^*(x), \  \tg(x,y)\leq0\}.
\eeq
Its associated Lagrangian function is given by
\begin{equation}\label{def-L}
\cL(x,y,\rho,\lambda)=f(x,y)+\rho(\tf(x,y)-\tf^*(x))+\langle\lambda,\tg(x,y)\rangle.
\end{equation}
In the spirit of classical constrained optimization, 
one would naturally be interested in a KKT solution $(x,y)$ of \eqref{prob-ref} or equivalently \eqref{prob},
namely, $(x,y)$ satisfies 
\beq \label{constr-cond}
\tf(x,y)\leq \tf^*(x), \quad \tg(x,y)\leq0, \quad \rho(\tf(x,y)- \tf^*(x))=0, \quad \langle \lambda, \tg(x,y) \rangle=0,
\eeq
and moreover $(x,y)$ is a stationary point of the problem
\beq \label{min-prob1}
\min_{x',y'}\cL(x',y',\rho,\lambda)
\eeq 
for some $\rho\geq 0$ and $\lambda\in\bR^l_+$.
 Yet, due to the sophisticated problem structure, characterizing a stationary point of  \eqref{min-prob1} is generally difficult. On another hand, notice from Lemma \ref{dual-bnd} and
\eqref{def-L} that problem \eqref{min-prob1} is equivalent to the minimax problem
\[
\min_{x',y',\tlambda'}\max_{z'} \big\{f(x',y')+\rho\big(\tf(x',y')-\tf(x',z')-\langle\tlambda',\tg(x',z')\rangle\big)+\langle\lambda,\tg(x',y')\rangle+\cI_{\bR^l_+}(\tlambda')\big\},
\]
whose stationary point $(x,y,\tlambda,z)$, according to Definition \ref{def2} and Assumptions \ref{a1} and \ref{a2}, satisfies 
\begin{align}
&0\in\partial f(x,y)+\rho\partial\tf(x,y) - \rho(\nabla_x\tf(x,z)+\nabla_x\tg(x,z)\tlambda;0)+\nabla \tg(x,y)\lambda, \label{kkt-c1}\\ 
&0\in\rho(\partial_z\tf(x,z)+\nabla_z\tg(x,z)\tlambda),  \label{kkt-c2} \\
& \tlambda \in \bR^l_+, \quad \tg(x,z) \leq 0, \quad \langle \tlambda, \tg(x,z)\rangle=0.\footnote{} \label{kkt-c3}
\end{align}
\footnotetext{The relations in \eqref{kkt-c3} are equivalent to $0 \in -\tg(x,z)+ \partial \cI_{\bR^l_+}(\tlambda)$.}
Based on this observation and also the fact that \eqref{constr-cond} is equivalent to 
\beq \label{constr-cond-new}
\tf(x,y)=\tf^*(x), \quad \tg(x,y)\leq0, \quad \langle \lambda, \tg(x,y) \rangle=0,
\eeq
we are instead interested in a (weak) KKT solution of problem \eqref{prob} and its inexact counterpart that are defined below.

\begin{defi} \label{KKT-soln}
The pair $(x,y)$ is said to be a KKT solution of problem \eqref{prob} if there exists $(z,\rho,\lambda,\tlambda)\in\bR^m\times\bR_+\times\bR^l_+\times\bR^l_+$ such that 
\eqref{kkt-c1}-\eqref{constr-cond-new} hold. In addition, for any $\varepsilon>0$, $(x,y)$ is said to be an $\varepsilon$-KKT solution of problem \eqref{prob} if there exists $(z,\rho,\lambda,\tlambda)\in\bR^m\times\bR_+\times\bR^l_+\times\bR^l_+$ such that
\begin{align*}
&\dist\left(0,\partial f(x,y)+\rho\partial\tf(x,y) - \rho(\nabla_x\tf(x,z)+\nabla_x\tg(x,z)\tilde\lambda; 0)+\nabla \tg(x,y)\lambda\right)\leq\varepsilon,\\
& \dist\left(0,\rho(\partial_z\tf(x,z)+\nabla_z\tg(x,z)\tlambda)\right)\leq\varepsilon, \\
& \|[\tg(x,z)]_+\| \leq \varepsilon, \quad |\langle \tlambda, \tg(x,z)\rangle| \leq\varepsilon, \\ 
&|\tf(x,y)-\tf^*(x)|\leq\varepsilon, \quad \|[\tg(x,y)]_+\|\leq\varepsilon, \quad |\langle \lambda, \tg(x,y) \rangle|\leq\varepsilon,
\end{align*}
where $\tf^*$ is defined in \eqref{tfstarx}.
\end{defi}

We are now ready to present an \emph{operation complexity} of Algorithm~\ref{alg4}, measured by the amount of evaluations of $\nabla f_1$, $\nabla\tf_1$, $\nabla\tg$ and proximal operator of $f_2$ and $\tf_2$, for finding an $\cO(\varepsilon)$-KKT solution of \eqref{prob}, whose proof is deferred to Subsection \ref{sec:proof4}.

\begin{thm}[{\bf Complexity of Algorithm \ref{alg4}}]\label{complexity}
Suppose that Assumptions~\ref{a1} and \ref{a2}  hold. 
Let $f^*$, $f$, $\tf$, $\tg$, $D_\bfx$, $D_\bfy$, $\tf_{\rm hi}$, $\tf_{\rm low}$, $f_{\rm low}$, $\tf^*$, $\tf^*_{\rm hi}$, and $\tg_{\rm hi}$ be defined in \eqref{unc-prob}, \eqref{DxDy}, \eqref{tfbnd}, \eqref{lower-bnd}, \eqref{tfstarx}, \eqref{def-tFx} and \eqref{dg-bnd}, $L_{\nabla f_1}$, $L_{\nabla \tf_1}$, $L_\tf$, $L_{\nabla \tg}$, $L_\tg$ and $G$ be given in Assumptions~\ref{a1} and \ref{a2}, $\varepsilon$, $\rho$, $\mu$, $x^0$, $y^0$ and $\ze$ be given in Algorithm~\ref{alg4}, and
\begin{align}
&\tilde\lambda=2\varepsilon^{-1}[\tg(\xe,\ze)]_+, \quad \hat\lambda=2\varepsilon^{-3}[\tg(\xe,\ye)]_+,\label{def-lambda}\\
&\tL=L_{\nabla f_1}+2\varepsilon^{-1} L_{\nabla\tf_1}+4\varepsilon^{-3} (\tg_{\rm hi}L_{\nabla \tg}+L_\tg^2),\label{tL}\\
&\tilde \alpha=\min\Big\{1,\sqrt{4\varepsilon/(D_\bfy\tL)}\Big\},\ \tilde \Cr=(2+\tilde\alpha^{-1})(D_\bfx^2+D_\bfy^2)\tL+\max\big\{\varepsilon/D_\bfy,\tilde \alpha\tL/4\big\}D_\bfy^2,\notag\\
&\widetilde C=\frac{4\max\{1/(2\tL),\min\{D_\bfy\varepsilon^{-1},4/(\tilde\alpha \tL)\}\}}{[(3\tL+\varepsilon/(2D_\bfy))^2/\min\{\tL,\varepsilon/(2D_\bfy)\}+ 3\tL+\varepsilon/(2D_\bfy)]^{-2}\varepsilon^5}\notag\\
&\ \ \ \ \ \times\left(\tilde \Cr+2\tilde\alpha^{-1}[f^*-f_{\rm low}+2\varepsilon^{-1}(\tf_{\rm hi}-\tf_{\rm low})+\varepsilon^{-3}\tg_{\rm hi}^2+\varepsilon D_\bfy/4+\tL (D_\bfx^2+D_\bfy^2)]\right),\notag\\
&\widetilde K=\left\lceil32(1+f(x^0,y^0)-f_{\rm low}+\varepsilon D_\bfy/4)\tL\varepsilon^{-2}+32\varepsilon^3\left(1+4D_\bfy^2\tL^2\varepsilon^{-2}\right)-1\right\rceil_+,\notag\\
&\widetilde N=\left(\left\lceil96\sqrt{2}\left(1+(24\tL+4\varepsilon/D_\bfy)\tL^{-1}\right)\right\rceil+2\right)\max\left\{2,\sqrt{D_\bfy\tL\varepsilon^{-1}}\right\}\notag\\
&\ \ \ \ \ \times[(\widetilde K+1)(\log \widetilde C)_++\widetilde K+1+2\widetilde K\log(\widetilde K+1)].\notag
\end{align}
Then Algorithm~\ref{alg4} outputs an approximate solution $(\xe,\ye)$ of \eqref{prob} satisfying
\begin{align}
&\dist\left(0,\partial f(\xe,\ye)+\rho\partial\tf(\xe,\ye)-\rho(\nabla_x\tf(\xe,\ze)+\nabla_x\tg(\xe,\ze)\tilde\lambda;0)+\nabla\tg(\xe,\ye)\hat \lambda\right)\leq\varepsilon,\label{gap1}\\
&\dist\left(0,\rho(\partial_z\tf(\xe,\ze)+\nabla_z\tg(\xe,\ze)\tilde\lambda)\right)\leq\varepsilon,\label{gap2}\\
&\left\|[\tg(\xe,\ze)]_+\right\|\leq \varepsilon^2 G^{-1}D_\bfy(\varepsilon^2+L_{\tf})/2,\label{gap3}\\
&|\langle\tilde\lambda,\tg(\xe,\ze)\rangle|\leq\varepsilon^2 G^{-2}D_\bfy^2(\rho^{-1}\epsilon+L_\tf)^2/2,\label{gap4}\\
&|\tf(\xe,\ye)-\tf^*(\xe)|\leq\max\Big\{\varepsilon\left(1+f(x^0,y^0)-f_{\rm low}+2\varepsilon^5(\tL^{-1}+4D_\bfy^2\tL\varepsilon^{-2})+D_\bfy \varepsilon/4 \right),\nn\\
&\qquad\qquad\qquad\qquad\qquad\qquad  \varepsilon^2G^{-2}D_\bfy^2L_\tf(\varepsilon^2+\varepsilon L_f+L_\tf)/2 \Big\},\label{gap5}\\
&\left\|[\tg(\xe,\ye)]_+\right\|\leq\varepsilon^2 G^{-1}D_\bfy(\varepsilon^2+\varepsilon L_f+L_\tf)/2,\label{gap6}\\
&|\langle\hat\lambda,\tg(\xe,\ye)\rangle|\leq\varepsilon G^{-2}D_\bfy^2(\varepsilon^2+\varepsilon L_f+L_\tf)^2/2, \label{gap7}
\end{align}
after at most $\widetilde N$ evaluations of $\nabla f_1$, $\nabla\tf_1$, $\nabla \tg$ and proximal operator of $f_2$ and $\tf_2$, respectively.
\end{thm}

\begin{rem}
One can observe from Theorem~\ref{complexity} that $\tL=\cO(\varepsilon^{-3})$, $\tilde\alpha=\cO(\varepsilon^2)$, $\tilde\Cr=\cO(\varepsilon^{-5})$, $\widetilde C=\cO(\varepsilon^{-23})$, $\widetilde K=\cO(\varepsilon^{-5})$, and $\widetilde N=\cO(\varepsilon^{-7}\log\varepsilon^{-1})$. As a result, Algorithm~\ref{alg4} enjoys an operation complexity of $\cO(\varepsilon^{-7}\log \varepsilon^{-1})$, measured by the amount of evaluations of $\nabla f_1$, $\nabla\tf_1$, $\nabla \tg$ and proximal operator of $f_2$ and $\tf_2$, for finding an $\cO(\varepsilon)$-KKT solution $(\xe,\ye)$ of \eqref{prob} satisfying
\begin{align*}
&\dist\left(0,\partial f(\xe,\ye)+\rho\partial\tf(\xe,\ye) - \rho(\nabla_x\tf(\xe,\ze)+\nabla_x\tg(\xe,\ze)\tilde\lambda; 0)+\nabla \tg(\xe,\ye)\hat\lambda\right)\leq\varepsilon,\\
& \dist\left(0,\rho(\partial_z\tf(\xe,\ze)+\nabla_z\tg(\xe,\ze)\tlambda)\right)\leq\varepsilon, \\
& \|[\tg(\xe,\ze)]_+\|=\cO(\varepsilon^2), \quad |\langle \tlambda, \tg(\xe,\ze)\rangle|=\cO(\varepsilon^2), \\ 
&|\tf(\xe,\ye)-\tf^*(\xe)|=\cO(\varepsilon), \quad \|[\tg(\xe,\ye)]_+\|=\cO(\varepsilon^2), \quad |\langle\hat \lambda, \tg(\xe,\ye) \rangle|=\cO(\varepsilon),
\end{align*}
where $\tf^*$ is defined in \eqref{tfstarx}, $\hat\lambda,\tlambda\in\bR_+^l$ are defined in \eqref{def-lambda}, $\ze$ is given in Algorithm \ref{alg4} and $\rho = \varepsilon^{-1}$.
\end{rem}

\section{Numerical results}\label{sec:exp}
In this section we conduct some preliminary experiments to test the performance of our proposed methods (Algorithms \ref{alg2} and \ref{alg4}) with dynamic update on penalty and tolerance parameters. All the algorithms are coded in Matlab and all the computations are performed on a desktop with a 3.60 GHz Intel i7-12700K 12-core processor and 32 GB of RAM.

\subsection{Unconstrained bilevel optimization with linear upper level and  quadratic lower level}
In this subsection, we consider unconstrained bilevel optimization with linear upper level and  quadratic lower level in the form of 
\begin{equation}\label{unc-linear}
\begin{array}{rl}
\min & c^Tx+d^Ty+\cI_{[-1,1]^n}(x)\\ [4pt]
\mbox{s.t.}& y\in\argmin\limits_{z}x^T\widetilde Az + z^T\widetilde Bz+\tilde d^Tz+\cI_{[-1,1]^m}(z), 
\end{array}
\end{equation}
where $\widetilde A\in\bR^{n\times m}$, $\widetilde B\in\bR^{m\times m}$, $c\in\bR^n$, $d, \tilde d\in\bR^m$, and $\cI_{[-1,1]^n}(\cdot)$ and $\cI_{[-1,1]^m}(\cdot)$ are the indicator functions of $[-1,1]^n$ and $[-1,1]^m$ respectively.\footnote{The notation $[-1,1]^n$ denotes the set $\{x\in\bR^n|x_i\in[-1,1],\ i=1,\ldots,n.\}$.}

For each pair $(n,m)$, we randomly generate $10$ instances of problem \eqref{unc-linear}. Specifically, we first randomly generate $c$, $d$ with all the entries independently chosen from the standard normal distribution, and $\widetilde A$ with all the entries independently chosen from a normal distribution with mean 0 and standard deviation $0.01$. We then randomly generate an orthogonal matrix $U$ by performing $U=\mathrm{orth}(\mathrm{randn}(m))$,  an $m\times m$ diagonal matrix $D$ with its diagonal entries independently chosen from a normal distribution with mean 0 and standard deviation $0.01$ and then projected to $\bR_+$, and set $\widetilde B=UDU^T$. In addition, we randomly generate $\hat y\in[-1,1]^m$ with all the entries independently chosen from a normal distribution with mean $0$ and standard deviation $0.1$ and then projected to $[-1,1]^m$, and choose $\tilde d$ such that $\hat y$ is an optimal solution for the lower level optimization of \eqref{unc-linear} with $x=0$.

Notice that \eqref{unc-linear} is a special case of \eqref{unc-prob} with $f(x,y)=c^Tx+d^Ty+\cI_{[-1,1]^n}(x)$ and $\tf(x,z)=x^T\widetilde Az + z^T\widetilde Bz+\tilde d^Tz+\cI_{[-1,1]^m}(z)$ and can be suitably solved by Algorithm \ref{alg2}. For the sake of efficiency, we implement a variant of Algorithm \ref{alg2} with dynamic update on penalty and tolerance parameters. Specifically, we set $\rho_k=5^{k-1}$, $\varepsilon_k=\rho_k^{-1}$ and 
$x_{\varepsilon_{-1}}=0$.  For each $k \geq 0$, we run Algorithm \ref{alg2} with $(\varepsilon, \rho)=(\varepsilon_k, \rho_k)$ and $(x_{\varepsilon_{k-1}},\tilde y_{\varepsilon_{k-1}})$ as the initial point to generate $(x_{\varepsilon_k},y_{\varepsilon_k})$, where 
$\tilde y_{\varepsilon_{k-1}}\in\argmin_z\tf(x_{\varepsilon_{k-1}},z)$ is found by CVX \cite{grant2014cvx}.
We terminate the process once $\varepsilon_{\bk} \leq 10^{-4}$ and 
$(x_{\varepsilon_{\bk}},y_{\varepsilon_{\bk}})$ satisfies
\begin{equation*}
\tf(x_{\varepsilon_{\bk}},y_{\varepsilon_{\bk}})-\min_z\tf(x_{\varepsilon_{\bk}},z)\leq10^{-4}
\end{equation*}
for some $\bk$ and output $(x_{\varepsilon_{\bk}},y_{\varepsilon_{\bk}})$ as an approximate solution of \eqref{unc-linear}, where the value of $\min_z\tf(x_{\varepsilon_{\bk}},z)$ is computed by CVX.

The computational results of the aforementioned variant of Algorithm \ref{alg2} for the instances randomly generated above are presented in Table \ref{t-unc-linear}. In detail, the values of $n$ and $m$ are listed in the first two columns. For each pair $(n, m)$, the average initial objective value $f(x_{\varepsilon_{-1}},\tilde y_{\varepsilon_{-1}})$ and the average final objective value 
$f(x_{\varepsilon_{\bk}},y_{\varepsilon_{\bk}})$  over $10$ random instances are given in the rest of columns. One can observe that the approximate solution 
$(x_{\varepsilon_{\bk}},y_{\varepsilon_{\bk}})$  found by this method significantly reduces objective function value compared to the initial point 
$(x_{\varepsilon_{-1}},\tilde y_{\varepsilon_{-1}})$. 
\begin{table}[H]
\centering
\begin{tabular}{cc||cc}
\hline
$n$&$m$&Initial objective value&Final objective value\\\hline
100&100&-0.35&-101.67\\
200&200&-0.53&-194.91\\
300&300&-0.48&-307.43\\
400&400&-0.44&-401.71\\
500&500&-0.05&-527.45\\
600&600&0.99&-644.53\\
700&700&0.49&-759.54\\
800&800&-1.23&-872.77\\
900&900&-2.07&-1004.27\\
1000&1000&-1.06&-1107.61\\\hline
\end{tabular}
\caption{Numerical results for problem \eqref{unc-linear}}
\label{t-unc-linear}
\end{table}

\subsection{Constrained bilevel linear optimization}
In this subsection, we consider constrained bilevel linear optimization in the form of
\begin{equation}\label{linear}
\begin{array}{rl}
\min & c^Tx+d^Ty+\cI_{[-1,1]^n}(x)\\ [4pt]
\mbox{s.t.}& y\in\argmin\limits_{z}\left\{\tilde d^Tz+\cI_{[-1,1]^m}(z)\big|\widetilde Ax+\widetilde Bz-\tilde b\leq0\right\}, 
\end{array}
\end{equation}
where $c\in\bR^n$, $d, \tilde d\in\bR^m$,  $\tilde b\in\bR^l$, $\widetilde A\in\bR^{l\times n}$, $\widetilde B\in\bR^{l\times m}$, and $\cI_{[-1,1]^n}(\cdot)$ and $\cI_{[-1,1]^m}(\cdot)$ are the indicator functions of $[-1,1]^n$ and $[-1,1]^m$ respectively.

For each triple $(n,m,l)$, we randomly generate $10$ instances of problem \eqref{linear}. Specifically,  we first randomly generate $c$ and $d$ with all the entries independently chosen from the standard normal distribution. We then randomly generate $\widetilde A$ and $\widetilde B$ with all the entries independently chosen from a normal distribution with mean $0$ and standard deviation $0.01$. In addition, we randomly generate $\hat y\in[-1,1]^m$ with all the entries independently chosen from a normal distribution with mean $0$ and standard deviation $0.1$ and then projected to $[-1,1]^m$ and choose $\tilde d$ and $\tilde b$ such that  $\hat y$ is an optimal solution of the lower level optimization of \eqref{linear} with $x=0$.

Notice that \eqref{linear} is a special case of \eqref{prob} with $f(x,y)= c^Tx+d^Ty+\cI_{[-1,1]^n}(x)$, $\tf(x,z)=\tilde d^Tz+\cI_{[-1,1]^m}(z)$ and $\tg(x,z)=\widetilde Ax+\widetilde Bz-\tilde b$ 
and  can be suitably solved by Algorithm \ref{alg4}. 
For the sake of efficiency, we implement a variant of Algorithm \ref{alg4} with dynamic update on penalty and tolerance parameters. Specifically, we set $\rho_k=5^{k-1}$, $\mu_k=\rho_k^2$, $\varepsilon_k=\rho_k^{-1}$ and $x_{\varepsilon_{-1}}=0$. For each $k \geq 0$, we run Algorithm \ref{alg4} with $(\varepsilon, \rho, \mu)=(\varepsilon_k, \rho_k, \mu_k)$ and $(x_{\varepsilon_{k-1}},\tilde y_{\varepsilon_{k-1}})$ as the initial point to generate $(x_{\varepsilon_k},y_{\varepsilon_k})$, where $\tilde y_{\varepsilon_{k-1}}$ satisfies $\tP_{\mu_k}(x_{\varepsilon_{k-1}},\tilde y_{\varepsilon_{k-1}})\leq\min_{z}\tP_{\mu_k}(x_{\varepsilon_{k-1}},z)+\varepsilon_k$ with $\tP_{\mu_k}$ being given in \eqref{def-tFmu}, which can be found by the accelerated proximal gradient method \cite{Nest13}. We terminate the process once $\varepsilon_{\bk} \leq 10^{-4}$ and 
$(x_{\varepsilon_{\bk}},y_{\varepsilon_{\bk}})$ satisfies
\begin{equation*}
\|[\tg(x_{\varepsilon_{\bk}},y_{\varepsilon_{\bk}})]_+\|\leq10^{-4},\quad \tf(x_{\varepsilon_{\bk}},y_{\varepsilon_{\bk}})-\tf^*(x_{\varepsilon_{\bk}})\leq10^{-4}
\end{equation*}
for some $\bk$ and output $(x_{\varepsilon_{\bk}},y_{\varepsilon_{\bk}})$ as an approximate solution of \eqref{linear}, where $\tf^*$ is defined in \eqref{tfstarx} and the value $\tf^*(x_{\varepsilon_{\bk}})$ is computed by CVX \cite{grant2014cvx}.

The computational results of the aforementioned variant of Algorithm \ref{alg4} for the instances randomly generated above are presented in Table \ref{t-linear}. In detail, the values of $n$, $m$ and $l$ are listed in the first three columns. For each triple $(n, m, l)$, the average initial objective value $f(x_{\varepsilon_{-1}},\hat y)$ with $\hat y$ being generated above\footnote{Note that $(x_{\varepsilon_{-1}},\tilde y_{\varepsilon_{-1}})$ may not be a feasible point of \eqref{linear}. Nevertheless, $(x_{\varepsilon_{-1}},\hat y)$ is a feasible point of \eqref{linear} due to $x_{\varepsilon_{-1}}=0$ and the particular way for generating instances of \eqref{linear}. Besides, \eqref{linear} can be viewed as an implicit optimization problem in terms of the variable $x$. It is thus reasonable to use $f(x_{\varepsilon_{-1}},\hat y)$ as the initial objective value for the purpose of comparison.} and the average final objective value $f(x_{\varepsilon_{\bk}},y_{\varepsilon_{\bk}})$ over $10$ random instances are given in the rest of the columns. One can observe that the approximate solution 
$(x_{\varepsilon_{\bk}},y_{\varepsilon_{\bk}})$  found by this method significantly reduces objective function value compared to the initial point 
$(x_{\varepsilon_{-1}},\hat y)$.

\begin{table}[H]
\centering
\begin{tabular}{ccc||cc}
\hline
$n$ & $m$&$l$ &Initial objective value&Final objective value\\\hline
100&100&5&-0.51&-34.83\\
200&200&10&-0.15&-121.41\\
300&300&15&1.56&-208.44\\
400&400&20&-0.04&-298.25\\
500&500&25&1.45&-384.77\\
600&600&30&0.75&-470.31\\
700&700&35&0.09&-568.26\\
800&800&40&-0.98&-629.61\\
900&900&45&1.21&-689.00\\
1000&1000&50&1.44&-781.79\\\hline
\end{tabular}
\caption{Numerical results for problem \eqref{linear}}
\label{t-linear}
\end{table}

\section{Proof of the main results}\label{sec:proof}
In this section we provide a proof of our main results presented in Sections \ref{unconstr-BLO} and \ref{constr-BLO}, which are particularly Theorems \ref{c1}-\ref{complexity}.

\subsection{Proof of the main results in Section~\ref{unconstr-BLO}}\label{sec:proof3}
In this subsection we prove Theorems \ref{c1}, \ref{hypergrad-thm} and \ref{unc-complexity}. We first establish a lemma below, which will be used to prove Theorem \ref{c1} subsequently.

\begin{lemma}\label{t1}
Suppose that Assumption \ref{a1} holds and $(\xe, \ye, \ze)$ is an $\epsilon$-optimal solution of problem \eqref{unc-mmax} for some $\epsilon>0$. Let $f$, $\tf$, $f^*$, $f_{\rm low}$ and $\rho$ be given in \eqref{unc-prob}, \eqref{lower-bnd} and \eqref{unc-mmax}, respectively. Then we have
\begin{equation*}
\tf(\xe,\ye)\leq\min_{z}\tf(\xe,z)+\rho^{-1}(f^*-f_{\rm low}+2\epsilon), \quad f(\xe,\ye)\leq f^*+2\epsilon.
\end{equation*}
\end{lemma}

\begin{proof}
Since $(\xe,\ye,\ze)$ is an $\epsilon$-optimal solution of \eqref{unc-mmax}, it follows from Definition \ref{def1} that 
\[
\max_{z}P_\rho(\xe,\ye,z)\leq P_\rho(\xe,\ye,\ze)+\epsilon, \quad P_\rho(\xe,\ye,\ze)\leq \min_{x,y}\max_{z}P_\rho(x,y,z)+\epsilon.
\]
Summing up these inequalities yields
\begin{equation}\label{unc-pgap}
\max_{z}P_\rho(\xe,\ye,z)\leq\min_{x,y}\max_{z}P_\rho(x,y,z)+2\epsilon.
\end{equation}
Let $(x^*,y^*)$ be an optimal solution of \eqref{unc-prob}. It then follows that $f(x^*,y^*)=f^*$ and 
$\tf(x^*,y^*)=\min_{z}\tf(x^*,z)$. By these and the definition of $P_\rho$ in \eqref{unc-mmax}, one has
\begin{equation*}
\max_{z}P_\rho(x^*,y^*,z)= f(x^*,y^*)+\rho(\tf(x^*,y^*)-\min_{z}\tf(x^*,z))=f(x^*,y^*)=f^*,
\end{equation*}
which implies that
\begin{equation}\label{unc-ineq}
\min_{x,y}\max_{z}P_\rho(x,y,z)\leq\max_{z}P_\rho(x^*,y^*,z)=f^*.
\end{equation}
It then follows from \eqref{unc-mmax}, \eqref{unc-pgap} and \eqref{unc-ineq} that
\begin{equation*}
f(\xe,\ye)+\rho(\tf(\xe,\ye)-\min_{z}\tf(\xe,z))\overset{\eqref{unc-mmax}}{=}\max_{z}P_\rho(\xe,\ye,z)\overset{\eqref{unc-pgap} \eqref{unc-ineq}}{\leq} f^*+2\epsilon,
\end{equation*}
which together with $\tf(\xe,\ye)-\min_{z}\tf(\xe,z)\geq 0$ implies that
\begin{equation*}
f(\xe,\ye)\leq f^*+2\epsilon,\quad\tf(\xe,\ye)\leq\min_{z}\tf(\xe,z)+\rho^{-1}\left(f^*-f(\xe,\ye)+2\epsilon\right).
\end{equation*}
The conclusion of this lemma directly follows from these and \eqref{lower-bnd}.
\end{proof}

We are now ready to prove Theorem~\ref{c1}.

\begin{proof}[\textbf{Proof of Theorem~\ref{c1}}]
Let $\{(x^k,y^k,z^k)\}$ be generated by Algorithm \ref{alg1} with $\lim_{k\to\infty}(\rho_k,\epsilon_k)=(\infty,0)$.  By considering a convergent subsequence if necessary, we assume without loss of generality that $\lim_{k\to\infty} (x^k,y^k)=(x^*,y^*)$. We now show that $(x^*,y^*)$ is an optimal solution of problem \eqref{unc-prob}. Indeed, since $(x^k,y^k,z^k)$ is an $\epsilon_k$-optimal solution of \eqref{unc-mmax} with $\rho=\rho_k$, it follows from Lemma \ref{t1} with $(\rho,\epsilon)=(\rho_k,\epsilon_k)$ and $(\xe,\ye)=(x^k,y^k)$ that 
 \begin{equation*}
\tf(x^k,y^k)\leq\min_{z}\tf(x^k,z)+\rho_k^{-1}(f^*-f_{\rm low}+2\epsilon_k),\quad f(x^k,y^k)\leq f^*+2\epsilon_k.
\end{equation*}
By the continuity of $f$ and $\tf$, $\lim_{k\to\infty} (x^k,y^k)=(x^*,y^*)$, $\lim_{k\to\infty}(\rho_k,\epsilon_k)=(\infty,0)$, and taking limits as $k\to\infty$ on both sides of the above relations, we obtain that $\tf(x^*,y^*)\leq\min_{z}\tf(x^*,z)$ and $f(x^*,y^*)\leq f^*$, which clearly imply that $y^*\in\argmin_z\tf(x^*,z)$ and $f(x^*,y^*)=f^*$. Hence, $(x^*,y^*)$ is an optimal solution of \eqref{unc-prob} as desired.
\end{proof}

We next prove Theorem~\ref{hypergrad-thm}. 
\begin{proof}[\textbf{Proof of Theorem~\ref{hypergrad-thm}}]
Since $(x,y)$ is an $\varepsilon$-KKT solution of problem \eqref{unc-prob} with its associated $\rho\geq \rho_0$, it follows from Definition 3 that there exists $z\in\bR^m$ such that
\begin{align}
&\|\nabla_x f(x,y)+\rho\nabla_x\tf(x,y)-\rho\nabla_x\tf(x,z)\|\leq\varepsilon,\label{e1}\\
&\|\nabla_y f(x,y)+\rho\nabla_y\tf(x,y)\|\leq\varepsilon,\label{e2}\\
&\rho\|\nabla_y\tf(x,z)\|\leq\varepsilon,\quad\tf(x,y)-\min_{z'}\tf(x,z')\leq\varepsilon.\label{e3}
\end{align}

Using \eqref{e1}, the triangle inequality,  and the assumptions that $x\in\Omega$, $\nabla f(x',\cdot)$ is $L_1$-Lipschitz continuous  and  $\nabla^2 \tf(x',\cdot)$ is $L_2$-Lipschitz continuous for all $x'\in\Omega$, we have
\begin{align}
&\|\nabla_x f(x,y^*(x))+\rho\nabla_{xy}^2\tf(x,y^*(x))(y-z)\|\nn\\
&\leq \|\nabla_x f(x,y)+\rho\nabla_x\tf(x,y)-\rho\nabla_x\tf(x,z)\|+\|\nabla_x f(x,y^*(x))-\nabla_xf(x,y)\|\nn\\
&\quad+\rho\|\nabla_x\tf(x,y^*(x))+\nabla_{xy}^2\tf(x,y^*(x))(y-y^*(x))-\nabla_x\tf(x,y)\|\nn\\
&\quad+\rho\|\nabla_x\tf(x,z)-\nabla_x\tf(x,y^*(x))-\nabla_{xy}^2\tf(x,y^*(x))(z-y^*(x))\|\nn\\
&\leq\varepsilon+L_1\|y-y^*(x)\|+\frac{\rho L_2}{2}\|y-y^*(x)\|^2+\frac{\rho L_2}{2}\|z-y^*(x)\|^2.\label{e4}
\end{align}
By \eqref{e2}, \eqref{e3} and a similar argument as for deriving \eqref{e4}, one has
\begin{align*}
&\|\nabla_yf(x,y^*(x))+\rho\nabla_{yy}^2\tf(x,y^*(x))(y-z)\|\nn\\
&\leq\|\nabla_y f(x,y^*(x))-\nabla_yf(x,y)\|+\|\nabla_y f(x,y)+\rho\nabla_y\tf(x,y)\|+\rho\|\nabla_y\tf(x,z)\|\nn\\
&\quad+\rho\|\nabla_y\tf(x,y^*(x))+\nabla_{yy}^2\tf(x,y^*(x))(y-y^*(x))-\nabla_y\tf(x,y)\|\nn\\
&\quad+\rho\|\nabla_y\tf(x,z)-\nabla_y\tf(x,y^*(x))-\nabla_{yy}^2\tf(x,y^*(x))(z-y^*(x))\|\nn\\
&\leq L_1\|y-y^*(x)\|+2\varepsilon+\frac{\rho L_2}{2}\|y-y^*(x)\|^2+\frac{\rho L_2}{2}\|z-y^*(x)\|^2. 
\end{align*}
Using this inequality,  \eqref{hypergrad}, \eqref{C} and \eqref{e4}, we obtain that 
\begin{align}
&\|\nabla \Phi(x)\|=\|\nabla_x f(x,y^*(x))-\nabla_{xy}^2\tf(x,y^*(x))[\nabla_{yy}^2\tf(x,y^*(x))]^{-1}\nabla_yf(x,y^*(x))\|\nn\\
&=\|\nabla_x f(x,y^*(x))+\rho\nabla_{xy}^2\tf(x,y^*(x))(y-z)\nn\\
&\ \ \ -\nabla_{xy}^2\tf(x,y^*(x))[\nabla_{yy}^2\tf(x,y^*(x))]^{-1}[\nabla_y f(x,y^*(x))+\rho\nabla_{yy}^2\tf(x,y^*(x))(y-z)]\|\nn\\
&\leq\|\nabla_x f(x,y^*(x))+\rho\nabla_{xy}^2\tf(x,y^*(x))(y-z)\|\nn\\
&\ \ \ +\|\nabla_{xy}^2\tf(x,y^*(x))[\nabla_{yy}^2\tf(x,y^*(x))]^{-1}\|\cdot\|\nabla_y f(x,y^*(x))+\rho\nabla_{yy}^2\tf(x,y^*(x))(y-z)\|\nn\\
&\leq(2C+1)\varepsilon+(C+1)\Big(L_1\|y-y^*(x)\|+\frac{\rho L_2}{2}\|y-y^*(x)\|^2+\frac{\rho L_2}{2}\|z-y^*(x)\|^2\Big).\label{e6}
\end{align}

Recall from the assumption that $x\in\Omega \subset \calN$ and $\tf(x',\cdot)$ is strongly convex with modulus $\sigma>0$ for all $x'\in\calN$. It follows from these,  \eqref{e3} and the definition of $y^*(x)$ in \eqref{Phi} that
\begin{equation}\label{y-dist}
\|y-y^*(x)\|^2\leq 2\sigma^{-1}\big(\tf(x,y)-\min_{z'}\tf(x,z')\big)\leq2\sigma^{-1}\varepsilon,
\end{equation}
which together with $x\in\Omega$ and \eqref{def-Mf} implies that $\|\nabla_y f(x,y)\|\leq \bar C$. Using this,  \eqref{e2}, \eqref{e3}, $x\in\Omega \subset \calN$, $\nabla_y\tf(x,y^*(x))=0$, and the assumption that $\tf(x',\cdot)$ is strongly convex with modulus $\sigma>0$ for all $x'\in\calN$, we have
\begin{align}
\|y-y^*(x)\|\leq\ &\sigma^{-1}\|\nabla_y\tf(x,y)-\nabla_y\tf(x,y^*(x))\|=\sigma^{-1}\|\nabla_y\tf(x,y)\|\nn\\
\leq\ &(\rho\sigma)^{-1}(\|\nabla_y f(x,y)+\rho\nabla_y\tf(x,y)\|+\|\nabla_y f(x,y)\|)\leq(\rho\sigma)^{-1}(\varepsilon+ \bar C), \label{e7} \\
\|z-y^*(x)\|\leq\ &\sigma^{-1}\|\nabla_y\tf(x,z)-\nabla_y\tf(x,y^*(x))\|=\sigma^{-1}\|\nabla_y\tf(x,z)\|\leq(\rho\sigma)^{-1}\varepsilon. \label{z-dist}
\end{align}
It then follows from \eqref{y-dist}, \eqref{e7} and the definition of $\theta$ in \eqref{C} that $\|y-y^*(x)\| \leq \theta$. By this, \eqref{e6} and \eqref{z-dist}, one can conclude that \eqref{hypergrad-bnd1} holds. In addition,  
in view of \eqref{C}, one has $\theta \leq \sqrt{2\sigma^{-1}\varepsilon}$ and 
\[
\rho\theta^2=\min\big\{\rho^{-1}\sigma^{-2}(\varepsilon+\bar C)^2,\ 2\rho\sigma^{-1}\varepsilon\big\} \leq  \min\big\{\rho^{-1}\sigma^{-2}(\varepsilon_0+\bar C)^2,\ 2\rho\sigma^{-1}\varepsilon\big\} \leq \sqrt{2}\sigma^{-3/2}(\varepsilon_0+\bar C)\sqrt{\varepsilon}.
\]
Using these inequalities,  \eqref{hypergrad-bnd1}, $\varepsilon \leq \varepsilon_0$ and $\rho \geq \rho_0$, we see that 
\eqref{hypergrad-bnd2} holds.
\end{proof}

We next prove Theorem~\ref{unc-complexity}. Before proceeding, we establish a lemma below, which will be used to prove Theorem~\ref{unc-complexity} subsequently.

\begin{lemma}\label{t2}
Suppose that Assumption \ref{a1} holds and $(\xe,\ye,\ze)$ is an $\epsilon$-stationary point of \eqref{unc-mmax}. Let $D_\bfy$, $f_{\rm low}$, $\tf$, $\rho$, and $P_\rho$ be given in  \eqref{DxDy}, \eqref{lower-bnd} and \eqref{unc-mmax}, respectively. Then we have
\begin{align*}
&\dist\Big(0,\partial f(\xe,\ye)+\rho\partial\tf(\xe,\ye)-(\rho\nabla_x\tf(\xe,\ze);0)\Big)\leq\epsilon,\quad\dist\big(0,\rho\partial_z\tf(\xe,\ze)\big)\leq\epsilon,\\
&\tf(\xe,\ye)\leq\min_{z}\tf(\xe,z)+\rho^{-1}(\max_zP_\rho(\xe,\ye,z)-f_{\rm low}).
\end{align*}
\end{lemma}

\begin{proof}
Since $(\xe,\ye,\ze)$ is an $\epsilon$-stationary point of  \eqref{unc-mmax}, it follows from Definition \ref{def2} that  
\begin{equation*}
\dist\big(0,\partial_{x,y} P_\rho(\xe,\ye,\ze)\big)\leq\epsilon,\quad\dist\big(0,\partial_{z} P_\rho(\xe,\ye,\ze)\big)\leq\epsilon.
\end{equation*}
Using these and the definition of $P_\rho$ in \eqref{unc-mmax}, we have
\[
\dist\Big(0,\partial f(\xe,\ye)+\rho\partial\tf(\xe,\ye)-(\rho\nabla_x\tf(\xe,\ze);0)\Big)\leq\epsilon,\quad\dist\big(0,\rho\partial_z\tf(\xe,\ze)\big)\leq\varepsilon.
\]
In addition, by \eqref{unc-mmax}, we have
\begin{equation*}
f(\xe,\ye)+\rho(\tf(\xe,\ye)-\min_{z}\tf(\xe,z))=\max_{z} P_\rho(\xe,\ye,z),
\end{equation*}
which along with \eqref{lower-bnd} implies that
\begin{equation*}
\tf(\xe,\ye)-\min_{z}\tf(\xe,z)\leq\rho^{-1}(\max_zP_\rho(\xe,\ye,z)-f_{\rm low}).
\end{equation*}
This completes the proof of this lemma.
\end{proof}

We are now ready to prove Theorem~\ref{unc-complexity}.

\begin{proof}[\textbf{Proof of Theorem~\ref{unc-complexity}}]
Observe from \eqref{P-rho} that problem \eqref{unc-mmax} can be viewed as 
\[
\min_{x,y}\max_{z} \left\{P_\rho(x,y,z)=\bh(x,y,z)+p(x,y)-q(z)\right\},
\]
where $\bh(x,y,z) = f_1(x,y)+\rho\tf_1(x,y)-\rho\tf_1(x,z)$, $p(x,y)=f_2(x)+\rho \tf_2(y)$, and $q(z)=\rho\tf_2(z)$. Hence, problem \eqref{unc-mmax} is in the form of \eqref{ap-prob} with $\bH=P_\rho$. By Assumption \ref{a1} and $\rho=\varepsilon^{-1}$, one can see that $\bh$ is $\hL$-smooth on its domain, where $\hL$ is given in \eqref{hL}. Also, notice from Algorithm \ref{alg2} that $\epsilon_0=\varepsilon^{3/2} \leq \varepsilon/2$ due to $\varepsilon\in(0,1/4]$. Consequently, Algorithm~\ref{mmax-alg2} can be suitably applied to problem \eqref{unc-mmax} with $\rho=\varepsilon^{-1}$ for finding an $\epsilon$-stationary point $(\xe,\ye,\ze)$ of it.

In addition, notice from Algorithm \ref{alg2} that $\tf(x^0,y^0)\leq\min_y\tf(x^0,y)+\varepsilon$.    
Using this, \eqref{unc-mmax} and $\rho=\varepsilon^{-1}$, we obtain 
\begin{equation}\label{t7-3}
\max_z P_\rho(x^0,y^0,z)=f(x^0,y^0)+\rho(\tf(x^0,y^0)-\min_z\tf(x^0,z))\leq f(x^0,y^0)+\rho\varepsilon= f(x^0,y^0)+1.
\end{equation}
By this and \eqref{upperbnd} with $\bH=P_{\rho}$, $\epsilon=\varepsilon$, $\epsilon_0=\varepsilon^{3/2}$, $\hat x^0=(x^0,y^0)$, $D_q=D_\bfy$, and $L_{\nabla \bh}=\hL$,  one has
\begin{align*}
\max_zP_\rho(\xe,\ye,z)\leq\ &\  \max_zP_\rho(x^0,y^0,z)+\varepsilon D_\bfy/4+2\varepsilon^3(\hL^{-1}+4D_\bfy^2\hL\varepsilon^{-2})\notag\\
\overset{\eqref{t7-3}}{\leq}&\ 1+f(x^0,y^0)+\varepsilon D_\bfy/4+2\varepsilon^3(\hL^{-1}+4D_\bfy^2\hL\varepsilon^{-2}).
\end{align*}
It then follows from this and Lemma~\ref{t2} with $\epsilon=\varepsilon$ and $\rho=\varepsilon^{-1}$ that $(\xe,\ye,\ze)$ satisfies \eqref{unc-gap1} and \eqref{unc-gap2}.

We next show that at most $\widehat N$ evaluations of $\nabla f_1$, $\nabla\tf_1$, and proximal operator of $f_2$ and $\tf_2$ are respectively performed in Algorithm \ref{alg2}. Indeed, by \eqref{tfbnd}, \eqref{lower-bnd} and \eqref{unc-mmax}, one has 
\begin{align}
&\min_{x,y}\max_z P_\rho(x,y,z)\overset{\eqref{unc-mmax}}{=} \min_{x,y}\{f(x,y)+\rho(\tf(x,y)-\min_z\tf(x,z))\}\geq\min_{(x,y)\in\mcX\times\mcY}f(x,y)\overset{\eqref{lower-bnd}}{=}f_{\rm low},\label{t7-1}\\
&\min_{(x,y,z)\in\mcX\times\mcY\times\mcY}P_\rho(x,y,z)\overset{\eqref{unc-mmax}}{=} \min_{(x,y,z)\in\mcX\times\mcY\times\mcY}\{f(x,y)+\rho(\tf(x,y)-\tf(x,z))\}\overset{\eqref{tfbnd}\eqref{lower-bnd}}{\geq}\  f_{\rm low}+\rho(\tf_{\rm low}-\tf_{\rm hi}).\label{t7-2}
\end{align}
For convenience of the rest proof, let
\beq \label{special-bH}
\bH=P_\rho, \quad \bH^*=\min_{x,y}\max_zP_\rho(x,y,z), \quad \bH_{\rm low}=\min\{P_\rho(x,y,z)|(x,y,z)\in\mcX\times\mcY\times\mcY\}.
\eeq 
In view of these, \eqref{unc-ineq}, \eqref{t7-3}, \eqref{t7-1}, \eqref{t7-2}, and $\rho=\varepsilon^{-1}$, we obtain that
\begin{align*}
&\max_z\bH(x^0,y^0, z)\overset{\eqref{t7-3}}{\leq} f(x^0,y^0)+1, \qquad f_{\rm low}\overset{\eqref{t7-1}}{\leq}\bH^*\overset{\eqref{unc-ineq}}\leq f^*,\\
&\bH_{\rm low}\overset{\eqref{t7-2}}{\geq} f_{\rm low}+\rho(\tf_{\rm low}-\tf_{\rm hi})=f_{\rm low}+\varepsilon^{-1}(\tf_{\rm low}-\tf_{\rm hi}).
\end{align*}
Using these and Theorem~\ref{mmax-thm} with $\epsilon=\varepsilon$, $\hat x^0=(x^0,y^0)$, $D_p=\sqrt{D_\bfx^2+D_\bfy^2}$, $D_q=D_\bfy$, $\epsilon_0=\varepsilon^{3/2}$, $L_{\nabla\bh}=\hL$, $\alpha=\hat\alpha$, $\Cr=\hat \Cr$, and $\bH$, $\bH^*$, $\bH_{\rm low}$ given in \eqref{special-bH}, we can conclude that Algorithm \ref{alg2} performs at most $\widehat N$ evaluations of $\nabla f_1$, $\nabla\tf_1$ and proximal operator of $f_2$ and $\tf_2$ respectively for finding an approximate solution $(\xe,\ye)$ of problem \eqref{unc-prob} satisfying \eqref{unc-gap1} and \eqref{unc-gap2}.
\end{proof}

%

\subsection{Proof of the main results in Section~\ref{constr-BLO}}\label{sec:proof4}
In this subsection we prove Theorems \ref{c3} and \ref{complexity}. 
Before proceeding, 
 we define
\beq
r=G^{-1}D_\bfy(\rho^{-1}\epsilon+L_{\tf}), \quad \cBr = \{\lambda\in\bR^l_+: \|\lambda\|\leq r\},\label{def-r}
\eeq
where $D_\bfy$ is defined in \eqref{DxDy}, $G$ is given in Assumption~\ref{a2}(iii), and $\epsilon$ and $\rho$ are given in Algorithm \ref{alg4}. In addition, one can observe from \eqref{tfstarx} and \eqref{def-tFmu} that
\beq\label{p-ineq}
\min_z\tP_\mu(x,z)\leq \tf^*(x)\qquad \forall x\in\mcX,
\eeq
which will be frequently used later. 

We next establish several technical lemmas that will be used to prove Theorem \ref{c3} subsequently.


\begin{lemma}\label{dual-bnd}
Suppose that Assumptions \ref{a1} and \ref{a2} hold. Let $D_\bfy$,  $L_\tf$, $G$, $\tf^*$, $\tf^*_{\rm hi}$ and $\cBr$ be given in \eqref{DxDy}, \eqref{tfstarx}, \eqref{def-tFx}, \eqref{def-r} and Assumption~\ref{a2}, respectively. Then the following statements hold.
\begin{enumerate}[label=(\roman*)]
\item $\|\lambda^*\|\leq G^{-1}L_{\tf} D_\bfy$ and $\lambda^*\in\cBr$ for all $\lambda^*\in\Lambda^*(x)$ and $x\in\mcX$, where $\Lambda^*(x)$ denotes the set of optimal Lagrangian multipliers of problem 
\eqref{tfstarx} for any $x\in\mcX$.
\item The function $\tf^*$ is Lipschitz continuous on $\mcX$ and $\tf^*_{\rm hi}$ is finite.
\item It holds that
\beq\label{fstar-ref}
\tf^*(x)=\max_{\lambda}\min_{z}\tf(x,z)+\langle\lambda,\tg(x,z)\rangle-\cI_{\bR_+^l}(\lambda) \qquad \forall x\in\mcX,
\eeq
where $\cI_{\bR_+^l}(\cdot)$ is the indicator function associated with $\bR_+^l$.
\end{enumerate}
\end{lemma}


\begin{proof}
(i) Let $x\in\mcX$ and $\lambda^*\in\Lambda^*(x)$ be arbitrarily chosen, and let $z^*\in\mcY$ be such that $(z^*,\lambda^*)$ 
is a pair of primal-dual optimal solutions of \eqref{tfstarx}. It then follows that
\[
z^*\in\argmin_{z} \tf(x,z)+\langle\lambda^*,\tg(x,z)\rangle, \quad \langle\lambda^*,\tg(x,z^*)\rangle=0, \quad \tg(x,z^*) \leq 0, \quad \lambda^* \geq 0.
\]
The first relation above yields
\[
\tf(x,z^*)+\langle\lambda^*,\tg(x,z^*)\rangle\leq\tf(x,\hat z_x)+\langle\lambda^*,\tg(x,\hat z_x)\rangle,
\]
where $\hat z_x$ is given in Assumption~\ref{a2}(iii). By this and $\langle\lambda^*,\tg(x,z^*)\rangle=0$, one has
\begin{equation*}
\langle\lambda^*,-\tg(x,\hat z_x)\rangle\leq\tf(x,\hat z_x)-\tf(x,z^*),
\end{equation*}
which together with $\lambda^*\geq 0$, \eqref{DxDy} and Assumption~\ref{a2} implies that
\beq \label{lambda-bnd}
G \sum_{i=1}^l\lambda_i^*\leq \langle\lambda^*,-\tg(x,\hat z_x)\rangle \leq \tf(x,\hat z_x)-\tf(x,z^*) \leq L_{\tf}\|\hat z_x-z^*\| \leq L_{\tf}D_\bfy,
\eeq
where the first inequality is due to Assumption~\ref{a2}(iii), and the third inequality follows from \eqref{DxDy} and $L_\tf$-Lipschitz continuity of $\tf$ (see Assumption \ref{a2}(i)). By \eqref{def-r}, \eqref{lambda-bnd} and $\lambda^*\geq 0$, we have $\|\lambda^*\|\leq\sum_{i=1}^l\lambda_i^*\leq G^{-1}L_{\tf} D_\bfy$ and $\lambda^*\in\cBr$.


(ii) Recall from Assumptions \ref{a1}(i) and \ref{a2}(iii) that $\tf(x,\cdot)$ and $\tg_i(x,\cdot), \ i=1,\ldots,l,$  are convex for any given  $x\in\mcX$.  Using this, \eqref{tfstarx} and the first statement of this lemma, we observe that 
\beq \label{tfstar-ref}
\tf^*(x)=\min_{z}\max_{\lambda\in\cBr}\tf(x,z)+\langle\lambda,\tg(x,z)\rangle \qquad \forall x\in\mcX.
\eeq
Notice from Assumption~\ref{a2} that $\tf$ and $\tg$ are Lipschitz continuous on their domain. Then it is not hard to observe that $\max\{\tf(x,z)+\langle\lambda,\tg(x,z)\rangle|\lambda\in\cBr\}$ is a Lipschitz continuous function of $(x,z)$ on its domain. By this and \eqref{tfstar-ref}, one can easily verify that $\tf^*$ is Lipschitz continuous on $\mcX$. In addition, the finiteness of $\tf^*_{\rm hi}$ follows from \eqref{def-tFx}, the continuity of $\tf^*$, and the compactness of $\mcX$.

(iii) One can observe from \eqref{tfstarx} that for all $x\in\mcX$,
\[
\tf^*(x)=\min_{z}\max_{\lambda}\tf(x,z)+\langle\lambda,\tg(x,z)\rangle-\cI_{\bR_+^l}(\lambda)\geq\max_{\lambda}\min_{z}\tf(x,z)+\langle\lambda,\tg(x,z)\rangle-\cI_{\bR_+^l}(\lambda)
\]
where the inequality follows from the weak duality. In addition, it follows from Assumption \ref{a1} that the domain of $\tf(x,\cdot)$ is compact for all $x\in\mcX$. By this, \eqref{tfstar-ref} and the strong duality, one has
\[
\tf^*(x)=\max_{\lambda\in\cBr}\min_{z}\tf(x,z)+\langle\lambda,\tg(x,z)\rangle-\cI_{\bR_+^l}(\lambda)\qquad\forall x\in\mcX,
\]
which together with the above inequality implies that \eqref{fstar-ref} holds.
\end{proof}

\begin{lemma}
Suppose that Assumptions \ref{a1} and \ref{a2} hold and that $(\xe, \ye, \ze)$ is an $\epsilon$-optimal solution of problem \eqref{mmax} for some $\epsilon>0$. Let $f_{\rm low}$, $f$, $\tP_\mu$, $f^*_\mu$, $\rho$ and $\mu$ be given in \eqref{lower-bnd}, \eqref{prob}, \eqref{def-tFmu}, \eqref{def-Fmu} and \eqref{mmax}, respectively. Then we have
\begin{equation}\label{gap}
\tP_\mu(\xe,\ye)\leq\min_z\tP_\mu(\xe,z)+\rho^{-1}(f_\mu^*-f_{\rm low}+2\epsilon), \qquad f(\xe,\ye)\leq f_\mu^*+2\epsilon.
\end{equation}
\end{lemma}

\begin{proof}
The proof follows from the same argument as the one for Lemma~\ref{t1} with $f^*$ and $\tf$ being replaced by $f^*_\mu$ and $\tP_\mu$, respectively.
\end{proof}

\begin{lemma}
Suppose that Assumptions \ref{a1}-\ref{errorbnd1} hold. Let $\tf_{\rm low}$, $f^*$, $\tf^*_{\rm hi}$, $f^*_\mu$ be defined in \eqref{tfbnd}, \eqref{prob}, \eqref{def-tFx} and \eqref{def-Fmu}, and $L_f$, $\omega$ and $\bar\theta$ be given in Assumptions \ref{a2} and \ref{errorbnd1}. Suppose that $\mu \geq (\tf^*_{\rm hi}-\tf_{\rm low})/\bar\theta^2$. Then we have
\begin{equation}\label{pena-dif}
f_\mu^*\leq f^*+L_f\omega\Big(\sqrt{\mu^{-1}(\tf^*_{\rm hi}-\tf_{\rm low})}\Big).
\end{equation}
\end{lemma}

\begin{proof}
Let $x\in\mcX$, $y\in\argmin_{z}\{\tf(x,z)|\tg(x,z)\leq0\}$ and $z^*\in\argmin_{z}\tP_\mu(x,z)$ be arbitrarily chosen.
One can easily see from \eqref{def-tFmu} and \eqref{p-ineq} that 
$\tf(x,z^*)+\mu\left\|[\tg(x,z^*)]_+\right\|^2\leq\tf^*(x)$, which together with \eqref{tfbnd} and \eqref{def-tFx} implies that
\begin{equation}\label{gnorm}
\left\|[\tg(x,z^*)]_+\right\|^2\leq\mu^{-1}(\tf^*_{\rm hi}-\tf_{\rm low}).
\end{equation}
Since $\mu \geq (\tf^*_{\rm hi}-\tf_{\rm low})/\bar\theta^2$, it follows from \eqref{gnorm} that 
$\|[\tg(x,z^*)]_+\| \leq \bar\theta$. By this relation, $y\in\argmin\limits_z\{\tf(x,z)|\tg(x,z)\leq0\}$ and Assumption~\ref{errorbnd1}, there exists some $\hz^*$ such that 
\beq \label{yz-dist}
\|y-\hz^*\|\leq \omega(\left\|[\tg(x,z^*)]_+\right\|), \qquad  \hz^*\in\argmin_z \left\{\tf(x,z)\big|\, \|[\tg(x,z)]_+\| \leq \|[\tg(x,z^*)]_+\|\right\}.
\eeq
In view of \eqref{def-tFmu}, $z^*\in\argmin_{z}\tP_\mu(x,z)$ and the second relation in \eqref{yz-dist},  one can observe that $\hz^*\in\argmin_{z}\tP_\mu(x,z)$, which along with \eqref{def-Fmu} yields $f(x,\hz^*) \geq f^*_\mu$. Also, using \eqref{yz-dist} and  $L_f$-Lipschitz continuity of $f$ (see Assumption \ref{a2}), we have
\begin{equation*}
f(x,y)-f(x,\hz^*)\geq-L_f\|y-\hz^*\|\overset{\eqref{yz-dist}}\geq-L_f\omega(\left\|[\tg(x,z^*)]_+\right\|).
\end{equation*}
Taking minimum over $x\in\mcX$ and $y\in\argmin_{z}\{\tf(x,z)|\tg(x,z)\leq0\}$ on both sides of this relation, and using \eqref{prob},      
 \eqref{gnorm}, $f(x,\hz^*) \geq f^*_\mu$ and the monotonicity of $\omega$, we can conclude that \eqref{pena-dif} holds. 
\end{proof}

\begin{lemma}\label{t3}
Suppose that Assumptions \ref{a1}-\ref{errorbnd1} hold. Let $\tf_{\rm low}$, $f_{\rm low}$, $f$, $\tf$, $f^*$, $\tf^*$, $\tf^*_{\rm hi}$, $\rho$ and $\mu$ be given in \eqref{tfbnd}, \eqref{lower-bnd}, \eqref{prob}, \eqref{tfstarx}, \eqref{def-tFx} and \eqref{mmax}, and $L_f$, $\omega$ and $\bar\theta$ be given in Assumptions \ref{a2} and \ref{errorbnd1}, respectively.
Suppose that $\mu \geq (\tf^*_{\rm hi}-\tf_{\rm low})/\bar\theta^2$ and $(\xe, \ye, \ze)$ is an $\epsilon$-optimal solution of problem \eqref{mmax} for some $\epsilon>0$. Then we have
\begin{align*}
&f(\xe,\ye)\leq f^*+L_f\omega\Big(\sqrt{\mu^{-1}(\tf^*_{\rm hi}-\tf_{\rm low})}\Big)+2\epsilon,\\
&\tf(\xe,\ye)\leq \tf^*(\xe)+\rho^{-1}\Big(f^*-f_{\rm low}+L_f\omega\Big(\sqrt{\mu^{-1}(\tf^*_{\rm hi}-\tf_{\rm low})}\Big)+2\epsilon\Big),\\
&\left\|[\tg(\xe,\ye)]_+\right\|^2\leq \mu^{-1}\Big(\tf^*(\xe)-\tf_{\rm low}+\rho^{-1}\Big(f^*-f_{\rm low}+L_f\omega\Big(\sqrt{\mu^{-1}(\tf^*_{\rm hi}-\tf_{\rm low})}\Big)+2\epsilon\Big)\Big).
\end{align*}
\end{lemma}

\begin{proof}
By \eqref{def-tFmu}, \eqref{p-ineq}, and the first relation in \eqref{gap}, one has
\[
\tf(\xe,\ye)+\mu\left\|[\tg(\xe,\ye)]_+\right\|^2\overset{\eqref{def-tFmu}}{=}\tP_\mu(\xe,\ye) \overset{\eqref{p-ineq} \eqref{gap}}{\leq}\tf^*(\xe)+\rho^{-1}(f_\mu^*-f_{\rm low}+2\epsilon).
\]
It then follows from this and \eqref{tfbnd} that
\[
\tf(\xe,\ye)\leq\tf^*(\xe)+\rho^{-1}(f_\mu^*-f_{\rm low}+2\epsilon), \quad \left\|[\tg(\xe,\ye)]_+\right\|^2\leq\mu^{-1}(\tf^*(\xe)-\tf_{\rm low}+\rho^{-1}(f_\mu^*-f_{\rm low}+2\epsilon)).
\]
In addition, recall from \eqref{gap} that $f(\xe,\ye)\leq f_\mu^*+2\epsilon$. The conclusion of this lemma then follows from these three relations and \eqref{pena-dif}.
\end{proof}

We are now ready to prove Theorem~\ref{c3}.

\begin{proof}[\textbf{Proof of Theorem~\ref{c3}}]
Let $\{(x^k,y^k,z^k)\}$ be generated by Algorithm \ref{alg3} with $\lim_{k\to\infty}(\rho_k,\mu_k,\epsilon_k)=(\infty,\infty,0)$. By considering a convergent subsequence if necessary, we assume without loss of generality that $\lim_{k\to\infty} (x^k,y^k)=(x^*,y^*)$. We now show that $(x^*,y^*)$ is an optimal solution of problem \eqref{prob}. Indeed, since $(x^k,y^k,z^k)$ is an $\epsilon_k$-optimal solution of \eqref{mmax} with $(\rho,\mu)=(\rho_k,\mu_k)$ and $\lim_{k\to\infty}\mu_k=\infty$, it follows from Lemma \ref{t3} with $(\rho,\mu,\epsilon)=(\rho_k,\mu_k,\epsilon_k)$ and $(\xe,\ye)=(x^k,y^k)$ that for all sufficiently large $k$, one has
\begin{align*}
&f(x^k,y^k)\leq f^*+L_f\omega\Big(\sqrt{\mu^{-1}_k(\tf^*_{\rm hi}-\tf_{\rm low})}\Big)+2\epsilon_k,\\
&\tf(x^k,y^k)\leq \tf^*(x^k)+\rho_k^{-1}\Big(f^*-f_{\rm low}+L_f\omega\Big(\sqrt{\mu^{-1}_k(\tf^*_{\rm hi}-\tf_{\rm low})}\Big)+2\epsilon_k\Big),\\
&\big\|[\tg(x^k,y^k)]_+\big\|^2\leq \mu_k^{-1}\Big(\tf^*(x^k)-\tf_{\rm low}+\rho_k^{-1}\Big(f^*-f_{\rm low}+L_f\omega\Big(\sqrt{\mu^{-1}_k(\tf^*_{\rm hi}-\tf_{\rm low})}\Big)+2\epsilon_k\Big)\Big).
\end{align*}
By the continuity of $f$, $\tf$ and $\tf^*$ (see Assumption \ref{a1}(i) and Lemma \ref{dual-bnd}(ii)), $\lim_{k\to\infty} (x^k,y^k)=(x^*,y^*)$, $\lim_{k\to\infty}(\rho_k,\mu_k,\epsilon_k)=(\infty,\infty,0)$, $\lim_{\theta\downarrow 0}\omega(\theta)=0$, and taking limits as $k\to\infty$ on both sides of the above relations, we obtain that  $f(x^*,y^*)\leq f^*$, $\tf(x^*,y^*)\leq\tf^*(x^*)$ and $[\tg(x^*,y^*)]_+=0$, which along with \eqref{prob} and \eqref{tfstarx} imply that $f(x^*,y^*)=f^*$ and $y^*\in\argmin_z\{\tf(x^*,z)|\tg(x^*,z)\leq 0\}$.  Hence, $(x^*, y^*)$ is an optimal solution of  \eqref{prob} as desired.
\end{proof}

We next prove Theorem~\ref{complexity}. Before proceeding, we establish several technical lemmas below, which will be used to prove Theorem~\ref{complexity} subsequently.

\begin{lemma}
Suppose that Assumptions \ref{a1} and \ref{a2} hold and that $(\xe,\ye,\ze)$ is an $\epsilon$-stationary point of problem \eqref{mmax} for some $\epsilon>0$. Let $D_\bfy$, $\tg$, $\rho$, $\mu$, $L_f$, $L_\tf$ and $G$ be given in \eqref{DxDy}, \eqref{prob}, \eqref{mmax} and Assumption \ref{a2}, respectively. Then we have
\begin{align}
&\left\|[\tg(\xe,\ze)]_+\right\| \leq (2\mu G)^{-1}D_\bfy(\rho^{-1}\epsilon+L_{\tf}), \label{gxz} \\
&\left\|[\tg(\xe,\ye)]_+\right\|\leq (2\mu G)^{-1}D_\bfy(\rho^{-1}\epsilon+\rho^{-1}L_f+L_\tf). \label{gxy}
\end{align}
\end{lemma}

\begin{proof}
We first prove \eqref{gxz}. Since $(\xe,\ye,\ze)$ is an $\epsilon$-stationary point of \eqref{mmax}, it follows from Definition \ref{def2} that $\dist(0,\partial_z P_{\rho,\mu}(\xe,\ye,\ze))\leq\epsilon$. Also, by \eqref{def-tFmu} and \eqref{mmax}, one has
\beq \label{P-rhomu}
P_{\rho,\mu}(x,y,z)=f(x,y)+\rho(\tf(x,y)+\mu\left\|[\tg(x,y)]_+\right\|^2)-\rho(\tf(x,z)+\mu\left\|[\tg(x,z)]_+\right\|^2).
\eeq
Using these relations, we have
\begin{equation*}
\dist\Big(0,\partial_z\tf(\xe,\ze)+2\mu\sum^l_{i=1}[\tg_i(\xe,\ze)]_+\nabla_z \tg_i(\xe,\ze)\Big)\leq\rho^{-1}\epsilon.
\end{equation*}
Hence, there exists $s\in\partial_z\tf(\xe,\ze)$ such that
\begin{equation}\label{l5-1}
\Big\|s+2\mu\sum_{i=1}^l[\tg_i(\xe,\ze)]_+\nabla_z \tg_i(\xe,\ze)\Big\|\leq\rho^{-1}\epsilon.
\end{equation}
Let $\hz_\xe$ and $G$ be given in Assumption \ref{a2}(iii). It then follows that $\hz_\xe\in\mcY$ and $-\tg_i(\xe,\hz_\xe)\geq G >0$ for all $i$. Notice that  $[\tg_i(\xe,\ze)]_+ \tg_i(\xe,\ze)\geq 0$ for all $i$ and $\|\ze-\hz_\xe\|\leq D_\bfy$ due to \eqref{DxDy}. Using these,  \eqref{l5-1}, and the convexity of $\tf(\xe,\cdot)$ and $\tg_i(\xe,\cdot)$ for all $i$, we have
\begin{align}
& \tf(\xe,\ze)-\tf(\xe,\hat z_\xe)+2\mu G\sum_{i=1}^l[\tg_i(\xe,\ze)]_+\leq
\tf(\xe,\ze)-\tf(\xe,\hat z_\xe)-2\mu\sum_{i=1}^l[\tg_i(\xe,\ze)]_+\tg_i(\xe,\hat z_\xe)\nn \\
& \leq \tf(\xe,\ze)-\tf(\xe,\hat z_\xe)+2\mu\sum_{i=1}^l[\tg_i(\xe,\ze)]_+(\tg_i(\xe,\ze)-\tg_i(\xe,\hat z_\xe))\nn \\
& \leq \langle s, \ze-\hat z_\xe\rangle+2\mu\sum_{i=1}^l[\tg_i(\xe,\ze)]_+\langle\nabla_z\tg_i(\xe,\ze),\ze-\hat z_\xe\rangle\nn \\
& = \langle s+2\mu\sum_{i=1}^l[\tg(\xe,\ze)]_+\nabla_z \tg_i(\xe,\ze),\ze-\hat z_\xe\rangle\leq\rho^{-1}D_\bfy\epsilon, \label{bound-ineq}
\end{align}
where the first inequality is due to $-\tg_i(\xe,\hz_\xe)\geq G$ for all $i$, the second inequality follows from $[\tg_i(\xe,\ze)]_+ \tg_i(\xe,\ze)\geq 0$ for all $i$, the third inequality is due to $s\in\partial_z\tf(\xe,\ze)$ and the convexity of $\tf(\xe,\cdot)$ and $\tg_i(\xe,\cdot)$ for all $i$, and the last inequality follows from \eqref{DxDy} and \eqref{l5-1}. In view of \eqref{DxDy}, \eqref{bound-ineq}, and $L_\tf$-Lipschitz continuity of $\tf(x,y)$ (see Assumption \ref{a2}), one has
\begin{align*}
\left\|[\tg(\xe,\ze)]_+\right\|\leq \sum_{i=1}^l[\tg_i(\xe,\ze)]_+\overset{\eqref{bound-ineq}}{\leq}&\ (2\mu G)^{-1}(\rho^{-1} D_\bfy\epsilon+\tf(\xe,\hat z_\xe)-\tf(\xe,\ze))\\ 
\leq &\ (2\mu G)^{-1}(\rho^{-1}D_\bfy\epsilon+L_{\tf}\|\hat z_\xe-\ze\|)
\overset{\eqref{DxDy}}{\leq}  (2\mu G)^{-1}D_\bfy(\rho^{-1}\epsilon+L_{\tf}).
\end{align*}
Hence, \eqref{gxz} holds.

We next prove \eqref{gxy}. Since $(\xe,\ye,\ze)$ is an $\epsilon$-stationary point of \eqref{mmax}, it follows from Definition \ref{def2} that 
$\dist(0,\partial_{x,y} P_{\rho,\mu}(\xe,\ye,\ze))\leq\epsilon$. 
In addition, notice from \eqref{Prhomu-ref} that $P_{\rho,\mu}$ is the sum of a smooth function and a possibly nonsmooth function that is separable with respect to $x$, $y$ and $z$.  Consequently, $\partial_{x,y}P_{\rho,\mu}=\partial_xP_{\rho,\mu}\times\partial_yP_{\rho,\mu}$, which together with $\dist(0,\partial_{x,y} P_{\rho,\mu}(\xe,\ye,\ze))\leq\epsilon$ implies that 
$\dist(0,\partial_yP_{\rho,\mu}(\xe,\ye,\ze))\leq\epsilon$. By this relation and \eqref{Prhomu-ref}, one has
\begin{align*}
\dist\big(0,\ \partial_y f(\xe,\ye)+\rho\partial_y\tf(\xe,\ye)+2\rho\mu\nabla_y\tg(\xe,\ye)[\tg(\xe,\ye)]_+\big)\leq\epsilon.
\end{align*}
Hence, there exists $s\in\partial_y f(\xe,\ye)$ and $\tilde s\in\partial_y\tf(\xe,\ye)$ such that
\begin{equation}\label{l11-1}
\left\|s+\rho\tilde s+2\rho\mu\nabla_y\tg(\xe,\ye)[\tg(\xe,\ye)]_+\right\|\leq\epsilon.
\end{equation}
Let $\bar\cA(\xe,\ye)=\{i|\tg_i(\xe,\ye)>0,1\leq i \leq l\}$, $\hz_\xe$ and $G$ be given in Assumption \ref{a2}(iii). It then follows that $\hz_\xe\in\mcY$ and $-\tg_i(\xe,\hz_\xe)\geq G >0$ for all $i$. Using these and the convexity of $\tg_i(\xe,\cdot)$ for all $i$, we have
\begin{align}
&\langle\nabla_y\tg(\xe,\ye)[\tg(\xe,\ye)]_+,\ye-\hz_\xe\rangle=\sum_{i\in\bar\cA(\xe,\ye)}\langle \nabla_y\tg_i(\xe,\ye),\ye-\hz_\xe\rangle [g_i(\xe,\ye)]_+ \nn\\
&\geq\sum_{i\in\bar\cA(\xe,\ye)}(\tg_i(\xe,\ye)-\tg_i(\xe,\hz_\xe))[\tg_i(\xe,\ye)]_+\nn\\
&\geq\sum_{i\in\bar\cA(\xe,\ye)}G[\tg_i(\xe,\ye)]_+=G\sum_{i=1}^l[\tg_i(\xe,\ye)]_+\geq G\left\|[\tg(\xe,\ye)]_+\right\|,\label{l11-2}
\end{align}
where the first inequality follows from the convexity of $\tg(\xe,\cdot)$ and the second inequality is due to $-\tg_i(\xe,\hz_\xe)\geq G$. 
It then follows from this, \eqref{l11-1} and \eqref{l11-2} that
\begin{align}
 D_\bfy\epsilon&\ \geq\left\|s+\rho\tilde s+2\rho\mu\nabla_y\tg(\xe,\ye)[\tg(\xe,\ye)]_+\right\|\cdot\|\ye-\hz_\xe\|\nn \\
&\ \geq\langle s+\rho\tilde s+2\rho\mu\nabla_y\tg(\xe,\ye)[\tg(\xe,\ye)]_+,\ye-\hz_\xe\rangle\nn\\
&\ =\langle s+\rho\tilde s,\ye-\hz_\xe\rangle+2\rho\mu\langle\nabla_y\tg(\xe,\ye)[\tg(\xe,\ye)]_+,\ye-\hz_\xe\rangle\nn \\
&\overset{\eqref{l11-2}}{\geq}-\left(\|s\|+\rho\|\tilde s\|\right)\|\ye-\hz_\xe\|+2\rho\mu G\left\|[\tg(\xe,\ye)]_+\right\|\nn \\
&\ \geq-(L_f+\rho L_\tf) D_\bfy+2\rho\mu G\left\|[\tg(\xe,\ye)]_+\right\|, \label{tg-ineq}
\end{align}
where the last inequality follows from  $\|\ye-\hz_\xe\|\leq D_\bfy$ and the fact that $\|s\| \leq L_f$ and $\|\tilde s\| \leq L_\tf$, which are due to \eqref{DxDy}, $s\in\partial_y f(\xe,\ye)$, $\tilde s\in\partial_y\tf(\xe,\ye)$ and Assumption \ref{a2}(i). 
By \eqref{tg-ineq}, one can immediately see that \eqref{gxy} holds.
\end{proof}

\begin{lemma}\label{t4}
Suppose that Assumptions \ref{a1} and \ref{a2} hold. Let $f$, $\tf$, $\tg$, $D_\bfy$, $f_{\rm low}$, $\tf^*$ and $P_{\rho,\mu}$ be given in \eqref{unc-prob}, \eqref{DxDy}, \eqref{lower-bnd}, \eqref{tfstarx} and \eqref{mmax}, $L_f$, $L_\tf$ and $G$ be given in Assumptions \ref{a1} and \ref{a2}, $(\xe,\ye,\ze)$ be an $\epsilon$-stationary point of \eqref{mmax} for some $\epsilon>0$, and
\begin{equation}
\tilde\lambda=2\mu[\tg(\xe,\ze)]_+, \quad \hat\lambda=2\rho\mu[\tg(\xe,\ye)]_+. \label{def-tlam}
\end{equation}
Then we have
\begin{align}
&\dist\left(\partial f(\xe,\ye)+\rho\partial\tf(\xe,\ye)-\rho(\nabla_x\tf(\xe,\ze)+\nabla_x\tg(\xe,\ze)\tilde\lambda;0)+\nabla\tg(\xe,\ye)\hat \lambda\right)\leq\epsilon, \label{rel1}\\
&\dist\left(0,\rho(\partial_z\tf(\xe,\ze)+\nabla_z\tg(\xe,\ze)\tilde\lambda)\right)\leq\epsilon,\label{rel2}\\
&\left\|[\tg(\xe,\ze)]_+\right\|\leq (2\mu G)^{-1}D_\bfy(\rho^{-1}\epsilon+L_{\tf}),\label{rel3}\\
&|\langle\tilde\lambda,\tg(\xe,\ze)\rangle|\leq(2\mu)^{-1}G^{-2}D_\bfy^2(\rho^{-1}\epsilon+L_\tf)^2,\label{rel4}\\
&|\tf(\xe,\ye)-\tf^*(\xe)|\leq\max\big\{\rho^{-1}(\max_zP_{\rho,\mu}(\xe,\ye,z)-f_{\rm low}),(2\mu)^{-1} G^{-2}D_\bfy^2L_\tf(\rho^{-1}\epsilon+\rho^{-1}L_f+L_\tf) \big\},\label{rel5}\\
&\left\|[\tg(\xe,\ye)]_+\right\|\leq(2\mu G)^{-1}D_\bfy(\rho^{-1}\epsilon+\rho^{-1}L_f+L_\tf),\label{rel6}\\
&|\langle\hat\lambda,\tg(\xe,\ye)\rangle|\leq(2\mu)^{-1}\rho G^{-2}D_\bfy^2(\rho^{-1}\epsilon+\rho^{-1}L_f+L_\tf)^2.\label{rel7}
\end{align}
\end{lemma}

\begin{proof}
Since $(\xe,\ye,\ze)$ is an $\epsilon$-stationary point of \eqref{mmax}, it easily follows from \eqref{P-rhomu}, \eqref{def-tlam} and Definition \ref{def2} that 
\eqref{rel1} and \eqref{rel2} hold.
Also, it follows from \eqref{gxz} and \eqref{gxy} that \eqref{rel3} and \eqref{rel6} hold. In addition, in view of \eqref{def-tlam}, \eqref{rel3} and \eqref{rel6}, one has 
\begin{align*}
&|\langle\tilde\lambda,\tg(\xe,\ze)\rangle|\overset{\eqref{def-tlam}}{=}2\mu\left\|[\tg(\xe,\ze)]_+\right\|^2\overset{\eqref{rel3}}{\leq}(2\mu)^{-1}G^{-2}D_\bfy^2(\rho^{-1}\epsilon+L_\tf)^2,\\
&|\langle\hat\lambda,\tg(\xe,\ye)\rangle|\overset{\eqref{def-tlam}}{=}2\rho\mu\left\|[\tg(\xe,\ye)]\|_+\right\|^2\overset{\eqref{rel6}}{\leq}(2\mu)^{-1}\rho G^{-2}D_\bfy^2(\rho^{-1}\epsilon+\rho^{-1}L_f+L_\tf)^2,
\end{align*}
and hence \eqref{rel4} and \eqref{rel7} hold. 
Also, observe from the definition of $P_{\rho,\mu}$ in \eqref{mmax} that
\[
\tP_\mu(\xe,\ye)-\min_{z}\tP_\mu(\xe,z) = \rho^{-1}(\max_zP_{\rho,\mu}(\xe,\ye,z)-f(\xe,\ye)).
\]
Using this, \eqref{lower-bnd},  \eqref{def-tFmu} and \eqref{p-ineq}, we obtain that
\begin{align}
\tf(\xe,\ye)+\mu\left\|[\tg(\xe,\ye)]_+\right\|^2\overset{\eqref{def-tFmu}}{=}\tP_\mu(\xe,\ye) =\ \ \ &\min_{z}\tP_\mu(\xe,z)+\rho^{-1}(\max_zP_{\rho,\mu}(\xe,\ye,z)-f(\xe,\ye))\nn\\
\overset{\eqref{lower-bnd} \eqref{p-ineq}}{\leq}&\tf^*(\xe)+\rho^{-1}(\max_zP_{\rho,\mu}(\xe,\ye,z)-f_{\rm low}).\label{l11-ineq}
\end{align}
On the other hand, let $\lambda^*\in\bR^l_+$ be an optimal Lagrangian multiplier of problem \eqref{tfstarx} with $x=\xe$. It then follows from Lemma \ref{dual-bnd}(i) that  $\|\lambda^*\|\leq G^{-1}L_\tf D_\bfy$. Using these and \eqref{rel6}, we have
\begin{align*}
\tf^*(\xe)=\ &\min_y \left\{\tf(\xe,y)+\langle \lambda^*, \tg(\xe,y)\rangle\right\}\leq\tf(\xe,\ye)+\langle \lambda^*, \tg(\xe,\ye)\rangle \nn\\ 
\leq\ &\tf(\xe,\ye)+ \|\lambda^*\| \|[\tg(\xe,\ye)]_+\|\leq\tf(\xe,\ye)+(2\mu)^{-1} G^{-2}D_\bfy^2L_\tf(\rho^{-1}\epsilon+\rho^{-1}L_f+L_\tf).
\end{align*}
By this and \eqref{l11-ineq}, one can see that \eqref{rel5}  holds.
\end{proof}

We are now ready to prove Theorem~\ref{complexity}.

\begin{proof}[\textbf{Proof of Theorem~\ref{complexity}}]
Observe from \eqref{Prhomu-ref} that problem \eqref{mmax} can be viewed as
\begin{equation*}
\min_{x,y}\max_z\left\{P_{\rho,\mu}(x,y,z) = \bh(x,y,z)+p(x,y)-q(z)\right\},
\end{equation*}
where $\bh(x,y,z) = f_1(x,y)+\rho\tf_1(x,y)+\rho\mu\left\|[\tg(x,y)]_+\right\|^2-\rho\tf_1(x,z)-\rho\mu\left\|[\tg(x,z)]_+\right\|^2$, $p(x,y)=f_2(x)+\rho \tf_2(y)$ and $q(z)=\rho\tf_2(z)$. Hence, problem \eqref{mmax} is in the form of \eqref{ap-prob} with $H=P_{\rho,\mu}$. By Assumption \ref{a1}, \eqref{dg-bnd}, \eqref{tg-bnd}, $\rho=\varepsilon^{-1}$ and $\mu=\varepsilon^{-2}$, one can see that $\bh$ is $\tL$-smooth on its domain, where $\tL$ is given in \eqref{tL}. Also, notice from Algorithm \ref{alg4} that $\epsilon_0=\varepsilon^{5/2}\leq \varepsilon/2=\epsilon/2$ due to $\varepsilon\in(0,1/4]$. Consequently, Algorithm~\ref{mmax-alg2} can be suitably applied to problem \eqref{mmax} with $\rho=\varepsilon^{-1}$ and $\mu=\varepsilon^{-2}$ for finding an $\epsilon$-stationary point $(\xe,\ye,\ze)$ of it.

In addition, notice from Algorithm \ref{alg4} that $\tP_\mu(x^0,y^0)\leq\min_y\tP_{\mu}(x^0,y)+\varepsilon$. Using this, \eqref{mmax} and $\rho=\varepsilon^{-1}$, we obtain
\begin{equation}\label{t8-3}
\max_z P_{\rho,\mu}(x^0,y^0,z)\overset{\eqref{mmax}}{=}f(x^0,y^0)+\rho(\tP_\mu(x^0,y^0)-\min_z\tP_\mu(x^0,z))\leq f(x^0,y^0)+\rho\varepsilon= f(x^0,y^0)+1.
\end{equation}
By this and \eqref{upperbnd} with $\bH=P_{\rho,\mu}$, $\epsilon=\varepsilon$, $\epsilon_0=\varepsilon^{5/2}$, $\hat x^0=(x^0,y^0)$, $D_q=D_\bfy$ and $L_{\nabla \bh}=\tL$,  one has
\begin{align*}
\max_zP_{\rho,\mu}(\xe,\ye,z)\leq\ &\  \max_zP_{\rho,\mu}(x^0,y^0,z)+\varepsilon D_\bfy/4+2\varepsilon^5(\tL^{-1}+4D_\bfy^2\tL\varepsilon^{-2})\notag\\
\overset{\eqref{t8-3}}{\leq}&\ 1+f(x^0,y^0)+\varepsilon D_\bfy/4+2\varepsilon^5(\tL^{-1}+4D_\bfy^2\tL\varepsilon^{-2}).
\end{align*}
It then follows from this and Lemma~\ref{t4} with $\epsilon=\varepsilon$, $\rho=\varepsilon^{-1}$ and $\mu=\varepsilon^{-2}$ that $(\xe,\ye,\ze)$ satisfies the relations \eqref{gap1}-\eqref{gap7}.

We next show that at most $\widetilde N$ evaluations of $\nabla f_1$, $\nabla\tf_1$, $\nabla \tg$ and proximal operator of $f_2$ and $\tf_2$ are respectively performed in Algorithm~\ref{alg4}. Indeed, by \eqref{tfbnd}, \eqref{lower-bnd}, \eqref{dg-bnd}, \eqref{def-tFmu} and \eqref{mmax}, one has
\begin{align}
&\min_{x,y}\max_z P_{\rho,\mu}(x,y,z) \overset{\eqref{mmax}}{=}\min_{x,y}\{f(x,y)+\rho(\tP_\mu(x,y)-\min_z\tP_\mu(x,z))\}\geq\min_{(x,y)\in\mcX\times\mcY}f(x,y)\overset{\eqref{lower-bnd}}{=} f_{\rm low},\label{t8-1}\\
&\min \{P_{\rho,\mu}(x,y,z)|(x,y,z)\in\mcX\times\mcY\times\mcY\}\overset{\eqref{mmax}}{=}\min\{f(x,y)+\rho(\tP_\mu(x,y)-\tP_\mu(x,z))|(x,y,z)\in\mcX\times\mcY\times\mcY\}\notag \\
&\overset{\eqref{def-tFmu}}{=}\min\{f(x,y)+\rho(\tf(x,y)+\mu \|[\tg(x,y)]_+\|^2-\tf(x,z)-\mu \|[\tg(x,z)]_+\|^2)|(x,y,z)\in\mcX\times\mcY\times\mcY\} \notag \\
&\geq f_{\rm low}+\rho(\tf_{\rm low}-\tf_{\rm hi})-\rho\mu\tg_{\rm hi}^2,\label{t8-2}
\end{align}
where the last inequality follows from \eqref{tfbnd}, \eqref{lower-bnd} and \eqref{dg-bnd}. In addition, let $(x^*,y^*)$ be an optimal solution of \eqref{prob}. It then follows that $f(x^*,y^*)=f^*$ and $[\tg(x^*,y^*)]_+= 0$. By these, \eqref{tfbnd}, \eqref{def-tFmu} and \eqref{mmax}, one has 
\begin{align}
\min_{x,y}\max_z P_{\rho,\mu}(x,y,z) \leq &\  \max_{z}P_{\rho,\mu}(x^*,y^*,z)\overset{\eqref{mmax}}{=}f(x^*,y^*)+\rho\left(\tP_\mu(x^*,y^*)-\min_z\tP_\mu(x^*,z)\right)\notag\\
\overset{\eqref{def-tFmu}}{=}&\ f(x^*,y^*)+\rho(\tf(x^*,y^*)+\mu\|[\tg(x^*,y^*)]_+\|^2-\min_z\{\tf(x^*,z)+\mu\|[\tg(x^*,z)]_+\|^2\})\notag\\
\overset{\eqref{tfbnd}}\leq&\ f^*+\rho(\tf_{\rm hi}-\tf_{\rm low}).\label{t8-0}
\end{align}
For convenience of the rest proof, let
\beq
\bH=P_{\rho,\mu},\quad \bH^*=\min_{x,y}\max_z P_{\rho,\mu}(x,y,z),\quad\bH_{\rm low}=\min\{P_{\rho,\mu}(x,y,z)|(x,y,z)\in\mcX\times\mcY\times\mcY\}.\label{special-bHrhomu}
\eeq
In view of these, \eqref{t8-3}, \eqref{t8-1}, \eqref{t8-2}, \eqref{t8-0}, $\rho=\varepsilon^{-1}$ and $\mu=\varepsilon^{-2}$, we obtain that 
\begin{align*}
&\max_z\bH(x^0,y^0, z)\overset{\eqref{t8-3}}{\leq} f(x^0,y^0)+1,\quad f_{\rm low}\overset{\eqref{t8-1}}{\leq}\bH^*\overset{\eqref{t8-0}}{\leq} f^*+\rho(\tf_{\rm hi}-\tf_{\rm low})=f^*+\varepsilon^{-1}(\tf_{\rm hi}-\tf_{\rm low}),\\
&\bH_{\rm low}\overset{\eqref{t8-2}}\geq f_{\rm low}+\rho(\tf_{\rm low}-\tf_{\rm hi})-\rho\mu\tg_{\rm hi}^2=f_{\rm low}+\varepsilon^{-1}(\tf_{\rm low}-\tf_{\rm hi})-\varepsilon^{-3}\tg_{\rm hi}^2.
\end{align*}
Using these and Theorem~\ref{mmax-thm} with $\epsilon=\varepsilon$, $\hat x^0=(x^0,y^0)$, $D_p=\sqrt{D_\bfx^2+D_\bfy^2}$, $D_q=D_\bfy$, $\epsilon_0=\varepsilon^{5/2}$, $L_{\nabla\bh}=\tL$, $\alpha=\tilde\alpha$, $\Cr=\tilde\Cr$, and $\bH$, $\bH^*$, $\bH_{\rm low}$ given in \eqref{special-bHrhomu}, we can conclude that Algorithm~\ref{alg4} performs at most $\widetilde N$ evaluations of $\nabla f_1$, $\nabla\tf_1$, $\nabla \tg$ and proximal operator of $f_2$ and $\tf_2$ for finding an approximate solution $(\xe,\ye)$ of problem \eqref{prob} satisfying \eqref{gap1}-\eqref{gap7}.
\end{proof}

\section{Concluding remarks}\label{sec:conclude}

For the sake of simplicity, first-order penalty methods are  proposed only for problem \eqref{BLO-1} in this paper. It would be interesting to extend them to problem \eqref{BLO} by using a standard technique (e.g., see \cite{nocedal1999numerical}) for handling the constraint $g(x,y) \leq 0$. This will be left for the future research. 

%

\appendix
\section{A first-order method for nonconvex-concave minimax problem}

In this part, we aim to find an $\epsilon$-stationary point of the nonconvex-concave minimax problem
\begin{equation}\label{ap-prob}
\bH^* = \min_x\max_y\left\{\bH(x,y)\coloneqq \bh(x,y)+p(x)-q(y)\right\},
\end{equation}
which has at least one optimal solution and satisfies the following assumptions.
\begin{assumption}\label{mmax-a}
\begin{enumerate}[label=(\roman*)]
\item $p:\bR^n\to\bR\cup\{\infty\}$ and $q:\bR^m\to\bR\cup\{\infty\}$ are proper convex functions and continuous on $\dom\,p$ and $\dom\,q$, respectively, and moreover, $\dom\,p$ and $\dom\,q$ are compact.
\item The proximal operators associated with $p$ and $q$ can be exactly evaluated.
\item $\bh$ is $L_{\nabla\bh}$-smooth on $\dom\,p\times\dom\,q$, and moreover, $\bh(x,\cdot)$ is concave for any $x\in\dom\,p$.
\end{enumerate}
\end{assumption}

Recently, an accelerated inexact proximal point smoothing (AIPP-S) scheme was proposed in \cite{kong2021accelerated} for finding an approximate stationary point of a class of minimax composite nonconvex optimization problems, which includes \eqref{ap-prob} as a special case. When applied to \eqref{ap-prob},  AIPP-S requires the exact solution of $\max_y \left\{\bh(x',y)-q(y)-\frac{1}{2\lambda}\|y-y'\|^2\right\}$ for any $\lambda>0$, $x'\in\bR^n$, and $y'\in\bR^m$. However,  $\bh$ is typically sophisticated and the \emph{exact} solution of such problem usually cannot be found. Consequently, AIPP-S is generally not implementable for \eqref{ap-prob}. In addition,  a first-order method was proposed in \cite{zhao2020primal} which enjoys a first-order oracle complexity of 
$\cO(\varepsilon^{-3}\log\varepsilon^{-1})$ for finding an $\epsilon$-primal stationary point $x'$ of \eqref{ap-prob} that satisfies 
\[
 \Big\|\lambda^{-1}(x'- \argmin_{x} \Big\{\max_y H(x,y)+\frac{1}{2\lambda}\|x-x'\|^2\Big\} \Big\|\leq \epsilon
\]
for some $0<\lambda<L_{\nabla\bh}^{-1}$. Yet, this method does not suit our needs since our aim is to find an $\epsilon$-stationary point of \eqref{ap-prob}  introduced in Definition \ref{def2}. In what follows, we present a first-order method proposed in \cite[Algorithm 2]{lu2023first} for finding such an $\epsilon$-stationary point of \eqref{ap-prob}.

For ease of presentation, we define
\begin{align}
&D_p=\max\{\|u-v\|\big|u,v\in\dom\,p\},\quad D_q=\max\{\|u-v\|\big|u,v\in\dom\,q\},\label{ap-D}\\
&H_{\rm low}=\min\{H(x,y)|(x,y)\in\dom\, p\times\dom\, q\}.\label{ap-H}
\end{align}

Given an iterate $(x^k,y^k)$, the first-order method \cite[Algorithm 2]{lu2023first} finds the next iterate $(x^{k+1},y^{k+1})$ by applying \cite[Algorithm 1]{lu2023first}, which is a slight modification of a novel optimal first-order method \cite[Algorithm 4]{kovalev2022first},  to the strongly-convex-strongly-concave minimax problem
\beq \label{hk}
\min_x\max_y\left\{\bh_k(x,y)=\bh(x,y)-\epsilon\|y-y^0\|^2/(4D_q)+L_{\nabla \bh}\|x-x^k\|^2\right\}.
\eeq

For ease of reference, we next present a modified optimal first-order method \cite[Algorithm 1]{lu2023first} in Algorithm \ref{mmax-alg1} below for solving the strongly-convex-strongly-concave minimax problem
\begin{equation}\label{ea-prob}
\min_{x}\max_{y}\left\{ \h(x,y)+p(x)-q(y)\right\},
\end{equation}
where 
$\h(x,y)$ is $\sigma_x$-strongly-convex-$\sigma_y$-strongly-concave and $L_{\nabla\h}$-smooth on $\dom\,p\times\dom\,q$ for some $\sigma_x,\sigma_y>0$. In Algorithm \ref{mmax-alg1},  the functions $\hat h$, $a^k_x$ and $a^k_y$ are defined as follows:
\begin{align*}
&\hat h(x,y)=\h(x,y)-\sigma_x\|x\|^2/2+\sigma_y\|y\|^2/2,\\
&a^k_x(x,y)=\nabla_x\hat h(x,y)+\sigma_x(x-\sigma_x^{-1}z^k_g)/2,\quad a^k_y(x,y)=-\nabla_y\hat h(x,y)+\sigma_y y+\sigma_x(y-y^k_g)/8,
\end{align*}
where $y^k_g$ and $z^k_g$ are generated at iteration $k$ of Algorithm \ref{mmax-alg1} below. 

\begin{algorithm}[H]
\caption{A modified optimal first-order method for problem \eqref{ea-prob}}
\label{mmax-alg1}
\begin{algorithmic}[1]
\REQUIRE $\tau>0$, $\bar z^0=z^0_f\in-\sigma_x\dom\,p$,\footnote{} $\bar y^0=y^0_f\in\dom\,q$, $(z^0,y^0)=(\bar z^0, \bar y^0)$,  $\bar \alpha=\min\left\{1,\sqrt{8\sigma_y/\sigma_x}\right\}$, $\eta_z=\sigma_x/2$, $\eta_y=\min\left\{1/(2\sigma_y),4/(\bar \alpha\sigma_x)\right\}$, $\beta_t=2/(t+3)$, $\zeta=\left(2\sqrt{5}(1+8L_{\nabla\h}/\sigma_x)\right)^{-1}$, $\gamma_x=\gamma_y=8\sigma_x^{-1}$, and $\hat\zeta=\min\{\sigma_x,\sigma_y\}/L_{\nabla \h}^2$.
\FOR{$k=0,1,2,\ldots$}
\STATE $(z^k_g,y^k_g)=\bar \alpha(z^k,y^k)+(1-\bar \alpha)(z^k_f,y^k_f)$.
\STATE $(x^{k,-1},y^{k,-1})=(-\sigma_x^{-1}z^k_g,y^k_g)$.
\STATE $x^{k,0}=\prox_{\zeta\gamma_xp}(x^{k,-1}-\zeta\gamma_x a^k_x(x^{k,-1},y^{k,-1}))$.
\STATE $y^{k,0}=\prox_{\zeta\gamma_y q}(y^{k,-1}-\zeta\gamma_y a^k_y(x^{k,-1},y^{k,-1}))$.
\STATE $b^{k,0}_x=\frac{1}{\zeta\gamma_x}(x^{k,-1}-\zeta\gamma_x a^k_x(x^{k,-1},y^{k,-1})-x^{k,0})$.
\STATE $b^{k,0}_y=\frac{1}{\zeta\gamma_y}(y^{k,-1}-\zeta\gamma_y a^k_y(x^{k,-1},y^{k,-1})-y^{k,0})$.
\STATE $t=0$.
\WHILE{\\ $\gamma_x\|a^k_x(x^{k,t},y^{k,t})+b^{k,t}_x\|^2+\gamma_y\|a^k_y(x^{k,t},y^{k,t})+b^{k,t}_y\|^2>\gamma_x^{-1}\|x^{k,t}-x^{k,-1}\|^2+\gamma_y^{-1}\|y^{k,t}-y^{k,-1}\|^2$\\~~}
\STATE $x^{k,t+1/2}=x^{k,t}+\beta_t(x^{k,0}-x^{k,t})-\zeta\gamma_x(a^k_x(x^{k,t},y^{k,t})+b^{k,t}_x)$.
\STATE $y^{k,t+1/2}=y^{k,t}+\beta_t(y^{k,0}-y^{k,t})-\zeta\gamma_y(a^k_y(x^{k,t},y^{k,t})+b^{k,t}_y)$.
\STATE $x^{k,t+1}=\prox_{\zeta\gamma_x p}(x^{k,t}+\beta_t(x^{k,0}-x^{k,t})-\zeta\gamma_x a^k_x(x^{k,t+1/2},y^{k,t+1/2}))$.
\STATE $y^{k,t+1}=\prox_{\zeta\gamma_y q}(y^{k,t}+\beta_t(y^{k,0}-y^{k,t})-\zeta\gamma_y a^k_y(x^{k,t+1/2},y^{k,t+1/2}))$.
\STATE $b^{k,t+1}_x=\frac{1}{\zeta\gamma_x}(x^{k,t}+\beta_t(x^{k,0}-x^{k,t})-\zeta\gamma_x a^k_x(x^{k,t+1/2},y^{k,t+1/2})-x^{k,t+1})$.
\STATE $b^{k,t+1}_y=\frac{1}{\zeta\gamma_y}(y^{k,t}+\beta_t(y^{k,0}-y^{k,t})-\zeta\gamma_y a^k_y(x^{k,t+1/2},y^{k,t+1/2})-y^{k,t+1})$.
\STATE $t \leftarrow t+1$.
\ENDWHILE
\STATE $(x^{k+1}_f,y^{k+1}_f)=(x^{k,t},y^{k,t})$.
\STATE $(z^{k+1}_f,w^{k+1}_f)=(\nabla_x\hat h(x^{k+1}_f,y^{k+1}_f)+b^{k,t}_x,-\nabla_y\hat h(x^{k+1}_f,y^{k+1}_f)+b^{k,t}_y)$.
\STATE $z^{k+1}=z^k+\eta_z\sigma_x^{-1}(z^{k+1}_f-z^k)-\eta_z(x^{k+1}_f+\sigma_x^{-1}z^{k+1}_f)$.
\STATE $y^{k+1}=y^k+\eta_y\sigma_y(y^{k+1}_f-y^k)-\eta_y(w^{k+1}_f+\sigma_yy^{k+1}_f)$.
\STATE $x^{k+1}=-\sigma_x^{-1}z^{k+1}$.
\STATE $\tx^{k+1}=\prox_{\hat\zeta p}(x^{k+1}-\hat\zeta\nabla_x\h(x^{k+1},y^{k+1}))$.
\STATE $\ty^{k+1}=\prox_{\hat\zeta q}(y^{k+1}+\hat\zeta\nabla_y\h(x^{k+1},y^{k+1}))$.
\STATE Terminate the algorithm and output $(\tx^{k+1},\ty^{k+1})$ if
\begin{equation*}
\|\hat\zeta^{-1}(x^{k+1}-\tx^{k+1},\ty^{k+1}-y^{k+1})-(\nabla \h(x^{k+1},y^{k+1})-\nabla \h(\tx^{k+1},\ty^{k+1}))\|\leq\tau.
\end{equation*}
\ENDFOR
\end{algorithmic}							
\end{algorithm}
\footnotetext{For convenience, $-\sigma_x\dom\,p$ stands for the set $\{-\sigma_x u|u\in\dom\,p\}$.}

We are now ready to present the first-order method \cite[Algorithm 2]{lu2023first} for finding an $\epsilon$-stationary point of \eqref{ap-prob} in Algorithm \ref{mmax-alg2} below.


\begin{algorithm}[H]
\caption{A first-order method for problem~\eqref{ap-prob}}
\label{mmax-alg2}
\begin{algorithmic}[1]
\REQUIRE $\epsilon>0$, $\epsilon_0\in(0,\epsilon/2]$, $(\hat x^0,\hat y^0)\in\dom\,p\times\dom\,q$, $(x^0,y^0)=(\hat x^0,\hat y^0)$,  and $\epsilon_k=\epsilon_0/(k+1)$.
\FOR{$k=0,1,2,\ldots$}
\STATE Call Algorithm~\ref{mmax-alg1} with $\h\leftarrow \bh_k$, $\tau \leftarrow \epsilon_k$, $\sigma_x\leftarrow L_{\nabla \bh}$, $\sigma_y\leftarrow \epsilon/(2D_q)$, $L_{\nabla \h}\leftarrow 3L_{\nabla \bh}+\epsilon/(2D_q)$, $\bar z^0=z^0_f\leftarrow-\sigma_x x^k$, $\bar y^0=y^0_f\leftarrow y^k$, and denote its output by $(x^{k+1},y^{k+1})$, where $h_k$ is given in \eqref{hk}.
\STATE Terminate the algorithm and output $(\xe,\ye)=(x^{k+1},y^{k+1})$ if
\begin{equation*}
\|x^{k+1}-x^k\|\leq\epsilon/(4L_{\nabla \bh}).
\end{equation*}
\ENDFOR
\end{algorithmic}
\end{algorithm}


The following theorem presents the iteration complexity of Algorithm \ref{mmax-alg2}, whose proof is given in \cite[Theorem 2]{lu2023first}.

\begin{thm}[{\bf Complexity of Algorithm \ref{mmax-alg2}}]\label{mmax-thm}
Suppose that Assumption~\ref{mmax-a} holds. Let $\bH^*$, $H$ $D_p$, $D_q$, and $\bH_{\rm low}$ be defined in \eqref{ap-prob}, 
\eqref{ap-D} and \eqref{ap-H}, $L_{\nabla \bh}$ be given in Assumption \ref{mmax-a}, $\epsilon$, $\epsilon_0$ and $x^0$ be given in Algorithm~\ref{mmax-alg2}, and 
\begin{align*}
\alpha=&\ \min\left\{1,\sqrt{4\epsilon/(D_qL_{\nabla \bh})}\right\},\\
\delta=&\ (2+\alpha^{-1})L_{\nabla \bh} D_p^2+\max\left\{\epsilon/D_q,\alpha L_{\nabla \bh}/4\right\}D_q^2,\\
K=&\ \left\lceil16(\max_y\bH(x^0,y)-\bH^*+\epsilon D_q/4)L_{\nabla \bh}\epsilon^{-2}+32\epsilon_0^2(1+4D_q^2L_{\nabla \bh}^2\epsilon^{-2})\epsilon^{-2}-1\right\rceil_+,\\
N=&\ \left(\left\lceil96\sqrt{2}\left(1+\left(24L_{\nabla \bh}+4\epsilon/D_q\right)L_{\nabla \bh}^{-1}\right)\right\rceil+2\right)\Big\{2,\sqrt{D_qL_{\nabla \bh}\epsilon^{-1}}\Big\}\notag\\
&\ \times\Bigg((K+1)\Bigg(\log\frac{4\max\left\{\frac{1}{2L_{\nabla \bh}},\min\left\{\frac{D_q}{\epsilon},\frac{4}{\alpha L_{\nabla \bh}}\right\}\right\}\left(\delta+2\alpha^{-1}(\bH^*-\bH_{\rm low}+\epsilon D_q/4+L_{\nabla \bh} D_p^2)\right)}{\left[(3L_{\nabla \bh}+\epsilon/(2D_q))^2/\min\{L_{\nabla \bh},\epsilon/(2D_q)\}+ 3L_{\nabla \bh}+\epsilon/(2D_q)\right]^{-2}\epsilon_0^2}\Bigg)_+\notag\\
&\ +K+1+2K\log(K+1) \Bigg).
\end{align*}
Then Algorithm~\ref{mmax-alg2} terminates and outputs an $\epsilon$-stationary point $(\xe,\ye)$ of \eqref{ap-prob} in at most $K+1$ outer iterations that satisfies 
\begin{equation}\label{upperbnd}
\max_y\bH(\xe,y)\leq \max_y\bH(\hat x^0,y)+\epsilon D_q/4+2\epsilon_0^2\left(L_{\nabla \bh}^{-1}+4D_q^2L_{\nabla \bh}\epsilon^{-2}\right).
\end{equation}
Moreover, the total number of evaluations of $\nabla \bh$ and proximal operator of $p$ and $q$ performed in Algorithm~\ref{mmax-alg2} is no more than $N$, respectively.
\end{thm}

\end{document}